\documentclass[a4paper,12pt]{article}

\usepackage{amscd}
\usepackage{amsfonts}
\usepackage{amsmath}
\usepackage{amssymb}
\usepackage{amsthm}
\usepackage{ascmac}
\usepackage[T1]{fontenc}
\usepackage{here}
\usepackage{mathrsfs}
\usepackage{slashbox}
\usepackage{txfonts}
\usepackage[all]{xy}

\allowdisplaybreaks

\theoremstyle{plain}
\newtheorem{thm}{Theorem}[section]
\newtheorem{lmm}[thm]{Lemma}
\newtheorem{prp}[thm]{Proposition}
\newtheorem{crl}[thm]{Corollary}
\theoremstyle{definition}

\newtheorem{rmk}[thm]{Remark}
\newtheorem{exm}[thm]{Example}

\def\ens#1{{\mathchoice{\left\{ #1 \right\}}{\{ #1 \}}{\{ #1 \}}{\{ #1 \}}}}
\def\set#1#2{{\mathchoice{\left\{ #1 \middle| #2 \right\}}{\{ #1 \mid #2 \}}{\{ #1 \mid #2 \}}{\{ #1 \mid #2 \}}}}
\def\r#1{\text{\rm #1}}
\def\t#1{\text{#1}}
\def\v#1{{\mathchoice{\left| #1 \right|}{| #1 |}{| #1 |}{| #1 |}}}
\def\n#1{{\mathchoice{\left\| #1 \right\|}{\| #1 \|}{\| #1 \|}{\| #1 \|}}}
\def\dbrack#1{[ \![ #1 ]\!]}

\def\ol#1{\overline{#1}{}}
\def\tl#1{\tilde{#1}{}}
\def\ul#1{\underline{#1}{}}
\def\wh#1{\widehat{#1}{}}

\newcommand{\im}{\r{im}}

\newcommand{\bC}{\mathbb{C}}

\newcommand{\bF}{\mathbb{F}}

\newcommand{\bL}{\mathbb{L}}

\newcommand{\bN}{\mathbb{N}}

\newcommand{\bQ}{\mathbb{Q}}
\newcommand{\bR}{\mathbb{R}}

\newcommand{\bZ}{\mathbb{Z}}

\newcommand{\cA}{\mathscr{A}}
\newcommand{\cB}{\mathscr{B}}
\newcommand{\cC}{\mathscr{C}}
\newcommand{\cD}{\mathscr{D}}

\newcommand{\cF}{\mathscr{F}}

\newcommand{\cM}{\mathscr{M}}

\newcommand{\cO}{\mathscr{O}}
\newcommand{\cP}{\mathscr{P}}

\newcommand{\cT}{\mathscr{T}}
\newcommand{\cU}{\mathscr{U}}

\newcommand{\rC}{\r{C}}

\newcommand{\rH}{\r{H}}

\newcommand{\rL}{\r{L}}
\newcommand{\rM}{\r{M}}

\newcommand{\rR}{\r{R}}

\newcommand{\C}{\bC}
\newcommand{\F}{\bF}
\newcommand{\N}{\bN}
\newcommand{\Q}{\bQ}
\newcommand{\R}{\bR}
\newcommand{\Z}{\bZ}

\newcommand{\Cp}{\mathbb{C}_p}

\newcommand{\Gm}{\mathbb{G}_{\r{m}}}
\newcommand{\Qp}{\mathbb{Q}_p}

\newcommand{\Zp}{\mathbb{Z}_p}

\newcommand{\Ab}{\r{Ab}}

\newcommand{\Alg}{\r{Alg}}

\newcommand{\Ch}{\r{Ch}}
\newcommand{\cl}{\r{cl}}

\newcommand{\Desc}{\r{Desc}}

\newcommand{\Hom}{\r{Hom}}
\newcommand{\uHom}{\ul{\Hom}}

\newcommand{\id}{\r{id}}

\newcommand{\Int}{\r{Int}}

\newcommand{\Mod}{\r{Mod}}

\newcommand{\op}{\r{op}}

\newcommand{\Set}{\r{Set}}

\newcommand{\Spec}{\r{Spec}}
\newcommand{\Spf}{\r{Spf}}

\newcommand{\Tot}{\r{Tot}}
\newcommand{\Tpl}{\r{Top}}
\newcommand{\triv}{\r{triv}}

\newcommand{\Cech}{$\check{\t{C}}$ech }

\newcommand{\Gelfand}{Gel'fand}

\newcommand{\Teichmuller}{Teichm\"uller }

\newcommand{\ad}{\r{ad}}
\newcommand{\Algiso}{\Alg^{\iso}}
\newcommand{\BAb}{\Ab^{\r{A}}}
\newcommand{\CH}{\r{CH}}
\newcommand{\Chba}{\Ch^{-}}
\newcommand{\Chisoba}{\Ch^{\iso -}}
\newcommand{\CO}{\r{CO}}
\newcommand{\cTd}{\cT_{\r{d}}}
\newcommand{\cTnd}{\cT_{\r{nd}}}
\newcommand{\Der}{\r{Der}}
\newcommand{\Derba}{\Der^{-}}
\newcommand{\Derisoba}{\Der^{\iso -}}

\newcommand{\ev}{\r{ev}}
\newcommand{\hyb}{\r{hyb}}
\newcommand{\iso}{\r{iso}}
\newcommand{\Modiso}{\Mod^{\iso}}
\newcommand{\NBAb}{\Ab^{\r{nA}}}
\newcommand{\rCbd}{\rC_{\r{bd}}}
\newcommand{\rCfin}{\rC_{\r{fin}}}
\newcommand{\TDCH}{\r{TDCH}}
\newcommand{\NSet}{\r{NSet}}

\title{Derived analytic geometry for $\Z$-valued functions Part I - Topological properties}
\author{Federico Bambozzi and Tomoki Mihara}
\date{}

\begin{document}

\maketitle

\begin{abstract}
We study the Banach algebras $\rC(X, R)$ of continuous functions from a compact Hausdorff topological space X to a Banach ring $R$ whose topology is discrete. We prove that the Berkovich spectrum of $\rC(X, R)$ is homeomorphic to $\zeta(X) \times \cM(R)$, where $\zeta(X)$ is the Banaschewski compactification of $X$ and $\cM(R)$ is the Berkovich spectrum of $R$. We study how the topology of the spectrum of $\rC(X, R)$ is related to the notion of homotopy Zariski open embedding used in derived geometry. We find that the topology of $\zeta(X)$ can be easily reconstructed from the homotopy Zariski topology associated to $\rC(X, R)$. We also prove some results about the existence of Schauder bases on $\rC(X, R)$ and a generalisation of the Stone--Weierstrass Theorem, under suitable hypotheses on $X$ and $R$.
\end{abstract}

\tableofcontents

\section{Introduction}
\label{Introduction}

\vspace{0.2in}
\noindent {\large \bf Background}
\vspace{0.1in}

\noindent 
The theory of commutative $\rC^*$-algebras, i.\ e.\ algebras of continuous functions from compact Hausdorff topological spaces $X$ to $\C$, is one of the most important topics of functional analysis. The richness of this theory comes from the structure of $\C$ as a complete valued field with respect to the standard absolute value. In the authors' recent work \cite{BM21}, it has been shown that the theory of $\rC^*$-algebras has a nice description in terms of the derived analytic geometry introduced in \cite{BK17}, \cite{BB16}, \cite{BBK19} and \cite{BK20}. Therefore, it is natural to ask what happens when the continuous $\C$-valued functions on $X$ are replaced by the Banach algebra $\rC(X, R)$ of continuous functions valued on other Banach rings. Already in \cite{BM21}, the case when $R$ is a complete non-Archimedean valued field has been considered and it has been checked that from the derived analytic geometry of the Banach algebra $\rC(X,R) \cong \rC(\zeta(X),R)$, one can reconstruct the topological space $\zeta(X)$, where $\zeta(X)$ is the Banaschewski compactification of $X$.

\vspace{0.2in}
\noindent {\large \bf Main results}
\vspace{0.1in}

\noindent 
In this work, we consider a case that is at the opposite extreme of the ones considered in \cite{BM21}. We study the Banach modules $\rC(X,M)$ for a Banach module $M$ whose topology induced by the norm is discrete. We call such an $M$ a {\it Banach module isolated at $0$}. In particular, we are interested in a Banach ring isolated at $0$. The main example of such a Banach ring is $\Z_\infty = (\Z, \v{\cdot}_\infty)$, i.\ e.\ $\Z$ equipped with the Euclidean absolute value. Such rings have been already considered, for example in \cite{Nob68}, and more recently in Lecture 5 of \cite{Sch19}, where they play a key role in the study of the notion of solid modules, but they have been usually considered just as algebras without considering the canonical Banach ring structure induced by the norm(s) on $\Z$. We will argue that this additional structure is a fundamental ingredient in the investigation of the algebras $\rC(X, R)$ and that these algebras should be considered to be objects whose study belongs both to geometry and to functional analysis as $\rC^*$-algebras do. Moreover, the ring $R$ can be more complicated than the complete valued fields considered in \cite{BM21}, and this has a strong impact on the properties of the algebras $\rC(X, R)$. Indeed, in the first five sections of this paper, we explain how to generalise the results of \cite{BM21} to the case of continuous functions valued in a Banach ring isolated at $0$. The strategies of the proofs are the same as the ones in \cite{BM21}, but almost all the lemmata have to be proved and almost all the basic computations have to be done in different ways. Whereas the last two sections prove some further properties of the Banach algebras $\rC(X, R)$. This paper is the first part of a two-part work. In the second part, we will study the morphisms $\rC(X, R) \to \rC(X, S)$ induced by a morphism of Banach ring $R \to S$ (where $S$ does not necessarily need to be isolated at $0$). 

\vspace{0.1in}
We notice that when the base ring $R$ is non-Archimedean, all our results can be presented in two different contexts. One context is when only non-Archimedean Banach $R$-modules are considered, and the other is when all Banach modules are allowed. We denote such a choice by the choice of the base category $\cC = \BAb$, in the first case, and $\cC = \NBAb$, in the second case. We will explain the reasons for this notation in \S \ref{Preliminaries}. Although most of the basic results of this paper are not influenced by this choice, the results in the last two sections are strongly influenced. More extreme is the situation of the results presented in the second part of this work, where we will be interested in understanding when a homotopy epimorphism $R \to S$ of Banach rings induces a homotopy epimorphism $\rC(X, R) \to \rC(X, S)$. We will see that not only that the Archimedean projective tensor product rarely satisfies the property that the canonical morphism $\rC(X, R) \wh{\otimes}_R S \to \rC(X, S)$ is an isomorphism, but also that, unexpectedly, the most basic localisations of the Banach ring $\Z_\infty$ are not homotopy epimorphisms. This is in stark contrast with the case of non-Archimedean Banach rings where one can prove that all the localisation used in analytic geometry are homotopy epimorphisms, even when they are not defined by a Koszul-regular sequence. See \cite{BK20} for more information on this topic and in particular loc.\ cit.\ Proposition 5.8 where examples of such non-Koszul-regular homotopy epimorphisms are considered in relation to the sheafiness problem for Banach rings. So, in these situations, these two choices lead to very different outcomes.

\vspace{0.2in}
\noindent {\large \bf Structure of the paper}
\vspace{0.1in}

\noindent 
In \S \ref{Preliminaries}, we recall basic notions from the theory of Banach rings and the theory of quasi-Abelian categories. The main purpose of this section is to fix the notation used in the rest of the paper that follows the notation used in \cite{BM21}. The reader is referred to \cite{BM21} \S 1 for a more detailed exposition of this fundamental material. In \S \ref{Berkovich spectrum for functions valued in discrete Banach rings}, we define Banach rings and Banach modules isolated at $0$, provide some fundamental examples, and prove some basic properties of the algebras of continuous functions from a topological space to a Banach module isolated at $0$. It is proved in Proposition \ref{extension property} that for any topological space $X$ and any Banach module $M$ isolated at $0$, the Banach module $\rCfin(X, M)$ of continuous functions $X \to M$ with finite image is naturally isomorphic to the space $\rC(\zeta(X), M)$ of continuous functions $\zeta(X) \to M$. From this, we deduce in Theorem \ref{Gelfand duality} \Gelfand's duality between the category of totally disconnected compact Hausdorff topological spaces $X$ and the category of Banach $R$-algebras $\rC(X, R)$ for any Banach ring $R$ isolated at $0$ whose Berkovich spectrum $\cM(R)$ is connected. In \S \ref{Homotopy epimorphisms for functions valued in discrete Banach rings}, we characterise closed topological embeddings between totally disconnected compact Hausdorff topological spaces in terms of homotopy epimorphisms, where the latter is one of the fundamental notions of derived geometry. More precisely, by \Gelfand's duality over $R$, the datum of a continuous map $f \colon X \to Y$ between totally disconnected compact Hausdorff topological spaces is equivalent to the datum of a morphism of Banach $R$-algebras $f^* \colon \rC(Y, R) \to \rC(X, R)$. In Theorem \ref{embedding corresponds to homotopy epimorphism}, we show that $f$ is a closed topological embedding if and only if $f^*$ is a homotopy epimorphism. This generalises one of the main results of \cite{BM21} and shows how the geometry of $X$ can be studied using the methods of derived geometry. In \S \ref{Tate Acyclicity for functions valued in discrete Banach rings}, we generalise in Theorem \ref{Tate's acyclicity for closed immersion} Tate's Acyclicity Theorem for the morphism of $\rC(X, R)$ induced by a finite covering by closed subsets. In more detail, we prove that a finite family of closed topological embeddings into the totally disconnected compact Hausdorff topological space $X$ set-theoretically covers $X$ if and only if the associated Tate--\Cech complex (resp.\ derived Tate--\Cech complex) is strictly exact (resp.\ a zero object). We notice that, since $R$ is isolated at $0$, Banach's Open Mapping Theorem does not necessarily hold for Banach $R$-modules. Therefore, in this situation, the strict exactness of a complex of Banach $R$-modules cannot be checked by proving only the algebraic exactness. This is one of the biggest difficulties which did not occur in the classical setting. In \S \ref{Effective descent for modules over Banach algebras of functions valued in discrete Banach rings}, we prove from these results effective descent both for Banach modules over $\rC(X, R)$ and for objects of the derived category of Banach modules over $\rC(X, R)$ with respect to the covers induced by closed topological embeddings of the Banaschewski compactification of $X$. 

\vspace{0.1in}
The results discussed up to this point are parallel to the results proved in \cite{BM21} in the case when $R$ is a complete non-trivially valued field. Although the proofs in the case when $R$ is isolated at $0$ differ in many details due to the fact that many fundamental theorems of functional analysis do not hold for Banach modules over $R$, the basic strategy is similar. On the other hand, the algebras $\rC(X, R)$ have a much richer structure than the algebras $\rC(X, k)$, for already the underlying ring of $R$ itself can be a quite complicated ring instead of a field. Indeed, we prove in \S \ref{Berkovich spectrum for functions valued in discrete Banach rings} that the Berkovich spectrum of $\rC(X, R)$ is homeomorphic to $\zeta(X) \times \cM(R)$. We note that $\cM(R)$ can be a quite complicated compact Hausdorff topological space in our setting, while the Berkovich spectrum of a complete valuation field is just a point. Therefore, in the last two sections, we study further properties of the algebras $\rC(X, R)$ that do not have analogues when $R$ is a complete valued field.

\vspace{0.1in}
In \S \ref{Homological algebra for discrete Banach modules}, we prove some basic properties of the category of Banach $R$-modules isolated at $0$, which will be denoted by $\Modiso_\cC(R)$. The main result of this section, Theorem \ref{absorbing law}, gives sufficient conditions on $(M_1,M_2) \in \Modiso_\cC(R)^2$ that ensure that the canonical map $\rC(X, M_1) \wh{\otimes}_R M_2 \to \rC(X, M_1 \wh{\otimes}_R M_2)$ is an isomorphism and also counterexamples to this statement. A direct consequence of this theorem is that if $x \in \cM(R)$ is a point such that $\set{c \in R}{x(c) = 0}$ is a maximal ideal, then $\rC(X, M) \wh{\otimes}_R \kappa(x) \cong \rC(X, M \wh{\otimes}_R \kappa(x))$ for all $M \in \Modiso_\cC(R)$, where $\kappa(x)$ denotes the residue field at $x$. In \S \ref{Schauder bases and Z-forms}, we deal with the existence of $R$-linear Schauder bases for the algebras $\rC(X, R)$. We recall that it is classically known for the case where $R$ is a complete valuation field that $\rC(X,R)$ admits $R$-linear Schauder bases if $R$ is non-Archimedean. As an analogue, we show in Proposition \ref{Cfin is free} for the case where $R$ is isolated at $0$ that $\rC(X,R)$ admits $R$-linear Schauder bases if $R$ is non-Archimedean. We then discuss the existence of a particular kind of Schauder bases that we call {\it generalised van der Put bases} and the difference between them and the Mahler bases. We conclude with a generalisation of the Stone--Weierstrass Theorem in Theorem \ref{non-Archimedean Stone--Weierstrass}.

\section{Preliminaries}
\label{Preliminaries}

We follow the conventions and the terminology on Banach rings and quasi-Abelian category theory from \cite{BM21}. For example, we always assume rings to be unital and commutative, and monoid objects of a symmetric monoidal category to be commutative. Although we briefly recall the conventions and the terminology in this section, it is better for the reader to refer to \cite{BM21} \S 1 for the detailed explanations on omitted facts such as the relation between Banach ring theory and the monoidal categories $\BAb$ and $\NBAb$ of Banach Abelian groups and non-Archimedean Banach Abelian groups that we introduce below. Moreover, when we consider a functor between two additive categories we will always tacitly assume that the functor is additive unless stated otherwise. We also use the same symbol to denote a category and its class of objects. Therefore, if $\cC$ is a category, the formula $X \in \cC$ means that $X$ is an object of $\cC$. We formally set $0^{-1} \coloneqq \infty$ and $\infty \cdot 0 = 0 \cdot \infty = 0$ in order to avoid a redundant case classification when we take the inverse of the operator norm of a bounded homomorphism which might be $0$.

\vspace{0.1in}
A {\it normed set} is a set $S$ equipped with a map $\v{\cdot}_{S} \colon S \to [0,\infty)$. For normed sets $S_0$ and $S_1$, we denote by $S_0 \odot S_1$ the normed set given as the set $S_0 \times S_1$ equipped with the map $\v{\cdot}_{S_0 \odot S_1} \colon S_0 \times S_1 \to [0,\infty)$ assigning $\v{s_0}_{S_0} \v{s_1}_{S_1}$ to each $(s_0,s_1) \in S_0 \times S_1$. For normed sets $S_0$ and $S_1$, a map $\phi \colon S_0 \to S_1$ is said to be {\it bounded} if its operator norm $\sup \set{\v{s}_{S_0}^{-1} \v{\phi(s)}_{S_1}}{s \in S_0} \in [0,\infty]$ is finite. We denote by $\NSet$ the category of normed sets and bounded maps, which naturally forms a symmetric monoidal category with respect to $\odot$.

\vspace{0.1in}
A {\it complete norm} on an Abelian group $M$ is a map $\n{\cdot} \colon M \to [0,\infty)$ such that the map $M^2 \to [0,\infty), \ (m_0,m_1) \mapsto \n{m_0 - m_1}$ is a complete metric on $M$. A {\it Banach Abelian group} is an Abelian group $M$ equipped with a complete norm $\n{\cdot}_M$ on it. A Banach Abelian group is said to be {\it non-Archimedean} if it is an ultrametric space with respect to the metric induced by the norm. We denote by $\BAb$ the category of Banach Abelian groups and bounded group homomorphisms, and by $\NBAb \subset \BAb$ the full subcategory of non-Archimedean Banach Abelian groups.

\vspace{0.1in}
Let $\cC$ denote $\BAb$ (resp.\ $\NBAb$). We equip $\cC$ with the symmetric monoidal structures given by the completed tensor product assigning to each $(M_0,M_1) \in \cC^2$ the completion $M_0 \wh{\otimes} M_1$ of $M_0 \otimes_{\Z} M_1$ with respect to the metric associated to the tensor seminorm $M_0 \otimes_{\Z} M_1 \to [0,\infty)$ assigning to each $m \in M_0 \otimes_{\Z} M_1$ the infimum of $\sum_{j=0}^{n} \n{(m_{0,j},m_{1,j})}_{M_0 \odot M_1}$ (resp.\ $\max_{j=0}^{n} \n{(m_{0,j},m_{1,j})}_{M_0 \odot M_1}$) for $(m_{0,j},m_{1,j})_{j=0}^{n} \in (M_0 \odot M_1)^{n+1}$ with $n \in \N$ and $\sum_{j=0}^{n} m_{0,j} \otimes m_{1,j} = m$.

\vspace{0.1in}
Let $R$ be a monoid object of $\cC$. We denote by $\Mod_{\cC}(R)$ the quasi-Abelian category of $R$-module objects of $\cC$ and $R$-linear bounded homomorphisms in $\cC$ (cf.\ \cite{BB16} Proposition 3.15 and \cite{BB16} Proposition 3.18). We equip $\Mod_{\cC}(R)$ with the closed symmetric monoidal structure $\wh{\otimes}_R$ induced by $\otimes$ and the internal hom functor $\uHom_{\Mod_{\cC}(R)}$ given by $\Hom_{\Mod_{\cC}(R)}$ equipped with the pointwise operations and the operator norm. The reader should be careful that the unitor $R \wh{\otimes}_R M \stackrel{\cong}{\to} M$ is not necessarily an isometry for an $M \in \Mod_{\cC}(R)$, because the scalar multiplication $R \wh{\otimes} M \to M$ is not necessarily contracting. We denote by $\Alg_{\cC}(R)$ the category of monoid objects of $\Mod_{\cC}(R)$.

\vspace{0.1in}
For an $M \in \Mod_{\cC}(R)$, we say that $M$ is {\it projective} if the external hom functor $\Hom_{\Mod_{\cC}(R)}(M,\cdot) \colon \Mod_{\cC}(R) \to \Ab$ is exact, or equivalently by the left exactness of $\Hom_{\Mod_{\cC}(R)}(M,\cdot)$, sends any strict epimorphism to a surjective map (cf.\ \cite{Sch99} Definition 1.3.18), and is {\it strongly flat} if the functor $(\cdot) \wh{\otimes}_R M \colon \Mod_{\cC}(R) \to \Mod_{\cC}(R)$ is strongly exact, or equivalently by the strong right exactness of $(\cdot) \wh{\otimes}_R M$, preserves the kernel of any morphism (cf.\ \cite{BK20} Definition 3.2).

\section{Berkovich spectrum for functions valued in discrete Banach rings}
\label{Berkovich spectrum for functions valued in discrete Banach rings}

In this section, we introduce Banach rings and Banach modules isolated at $0$, and we study the basic properties of the modules of continuous functions from a topological space to such modules. We prove in Proposition \ref{extension property} that for any topological space $X$, any Banach ring $R$, and any Banach $R$-module $M$ isolated at $0$, the Banach $R$-module $\rCfin(X, M)$ of continuous functions $X \to M$ with finite image is canonically isometrically isomorphic to the Banach $R$-module $\rC(\zeta(X), M)$ of continuous functions $\zeta(X) \to M$, where $\zeta(X)$ denotes the Banaschewski compactification of $X$. Then, we deduce in Theorem \ref{Banaschweski compatification} that the Berkovich spectrum of $\rC(\zeta(X), M)$ is homeomorphic to $\zeta(X) \times \cM(R)$, where $\cM(R)$ denotes the Berkovich spectrum of $R$. We conclude this section by proving \Gelfand's duality between the category of totally disconnected compact Hausdorff topological spaces $X$ and the category of Banach $R$-algebras $\rC(X, R)$ in Theorem \ref{Gelfand duality}, under the hypothesis that $\cM(R)$ is connected.

\vspace{0.1in}
Let $\cC$ be either $\BAb$ or $\NBAb$, $R$ a monoid object of $\cC$, and $M \in \Mod_{\cC}(R)$. We say that $M$ is {\it isolated at $0$} if $\ens{0} \subset M$ is an open subset for the topology induced by the norm. We say that $R$ is {\it isolated at $0$} if so is the regular $R$-module object. Although we focus on the case where $M$ is the regular $R$-module object in \S \ref{Homotopy epimorphisms for functions valued in discrete Banach rings} -- \S \ref{Effective descent for modules over Banach algebras of functions valued in discrete Banach rings}, we will deal with the general case in \S \ref{Homological algebra for discrete Banach modules} and \S \ref{Schauder bases and Z-forms}.

\begin{exm}
\label{example of Banach ring isolated at 0}
\begin{itemize}
\item[(i)] Let $O \subset K$ be the ring of algebraic integers of a number field $K$. We denote by $\v{\cdot}_{\infty} \colon O \to [0,\infty)$ the norm given by the maximum of the restrictions of the norm of $(\C, \v{\cdot}_\infty)$ to $K$ via all infinite places $K \to \C$. For any $\epsilon \in [0,1]$, $(O,\v{\cdot}_{\infty}^{\epsilon})$ is a monoid of $\BAb$ isolated at $0$. In particular, $\Z_{\infty} \coloneqq (\Z,\v{\cdot}_{\infty})$ is a monoid of $\BAb$ isolated at $0$.
\item[(ii)] Let $K$ be a number field. We denote by $\v{\cdot}_{\ad} \colon K \to [0,\infty)$ the norm given by the supremum of the restrictions of the absolute values corresponding to all the places on $K$, where the absolute values $\v{\cdot}_v$ corresponding to finite places $v$ over a prime number $p$ are normalised by $\v{p}_v = p^{-1}$. For any $\epsilon \in [0,1]$, $(K,\v{\cdot}_{\ad}^{\epsilon})$ is a monoid of $\BAb$ isolated at $0$. In particular, $(\Q,\v{\cdot}_{\ad})$ is a monoid of $\BAb$ isolated at $0$. 
\item[(iii)] Let $A$ be a ring, and $N$ be an $A$-module. We denote by $\n{\cdot}_{\triv}$ the trivial norms on $A$ and $N$, which assign $1$ to each non-zero element and $0$ to the zero elements. Then, $A_{\triv} \coloneqq (A,\n{\cdot}_{\triv})$ is a monoid of $\NBAb$ isolated at $0$, and $N_{\triv} \coloneqq (N,\n{\cdot}_{\triv})$ is an object of $\Mod_{\NBAb}(A_{\triv})$ isolated at $0$. In particular, $\Z_{\triv}$ is a monoid of $\NBAb$ isolated at $0$. 
\item[(iv)] Let $A$ be a monoid of $\cC$. We denote by $\n{\cdot}_{\hyb}$ the hybrid norm on $A$, which assigns $\max \ens{\n{a}_A,\n{a}_{\triv}}$ to each $a \in A$. Then, $A_{\hyb} \coloneqq (A,\n{\cdot}_{\hyb})$ is a monoid of $\cC$ isolated at $0$. In particular, $\C_{\hyb}$ is a monoid of $\BAb$ isolated at $0$.
\end{itemize}
\end{exm}

We study elementary functional analysis on Banach modules isolated at $0$.

\begin{prp}
\label{isolated at 0 implies discrete}
The topology of $M$ is discrete if and only if $M$ is isolated at $0$.
\end{prp}

\begin{proof}
If the topology of $M$ is discrete, then every subset of $M$ (and hence $\ens{0} \subset M$) is open. Suppose that $M$ is isolated at $0$. Since $\ens{0} \subset M$ is open, there exists an $\epsilon \in (0,\infty)$ such that $\ens{0} = \set{m \in M}{\n{m}_M < \epsilon}$. For any $m \in M$, we have $\ens{m} = \set{m+m'}{m' \in M \land \n{m'} < \epsilon} = \set{m' \in M}{\n{m'-m}_M < \epsilon}$, and hence $\ens{m}$ is an open subset of $M$.
\end{proof}

Throughout this paper, let $X$ denote a topological space. We denote by $\rC(X,M)$ the $R$-module of continuous maps $X \to M$, by $\rCbd(X,M) \subset \rC(X,M)$ the $R$-submodule of continuous maps whose images are bounded, and by $\rCfin(X,M) \subset \rCbd(X,M)$ the $R$-submodule of continuous maps whose images are finite sets. Then, $\rCbd(X,M)$ forms an object of $\Mod_{\cC}(R)$ with respect to the supremum norm, while $\rCfin(X,M)$ is not necessarily closed in $\rCbd(X,M)$. When $X$ is compact, then $\rCbd(X,M)$ coincides with $\rC(X,M)$, and hence we also regard $\rC(X,M)$ as an object of $\Mod_{\cC}(R)$. Applying this construction to the regular $R$-module object, we obtain $R$-algebras $\rCfin(X,R) \subset \rCbd(X,R) \subset \rC(X,R)$. Then, $\rCfin(X,M)$ forms a $\rCfin(X,R)$-module, and $\rCbd(X,R)$ forms an object of $\Alg_{\cC}(R)$.

\begin{prp}
\label{Cfin is closed}
\begin{itemize}
\item[(i)] If $M$ is isolated at $0$, then every $R$-submodule of $\rCbd(X,M)$ is closed, and is isolated at $0$ as an object of $\Mod_{\cC}(R)$.
\item[(ii)] If $X$ is compact and $M$ is isolated at $0$, then $\rCfin(X,M)$ coincides with $\rC(X,M)$.
\end{itemize}
\end{prp}

\begin{proof}
The assertion (i) follows from the fact that every Cauchy sequence in $\rCbd(X,M)$ is eventually constant because $M$ is isolated at $0$. The assertion (ii) immediately follows from Proposition \ref{isolated at 0 implies discrete}.
\end{proof}

We denote by $\CO(X)$ the set of clopen subsets of $X$. For each $U \in \CO(X)$, we denote by $1_{X,R,U} \in \rCfin(X,R)$ the characteristic function of $U$. For each $(f,m) \in \rCfin(X,R) \times M$, we denote by $fm \in \rCfin(X,M)$ the scalar multiplication of $f$ and the constant map $X \to M$ with value $m$. We show a finite image analogue of \cite{Mih14} Theorem 4.12.

\begin{prp}
\label{idempotent generates C(X,R)}
\begin{itemize}
\item[(i)] The set $\set{1_{X,R,U} m}{(U,m) \in \CO(X) \times M}$ generates the underlying $R$-module of $\rCfin(X,M)$.
\item[(ii)] For any $f \in \rCfin(X,M)$, there exists a pair $(\cU,e)$ of a finite subset $\cU \subset \CO(X) \setminus \ens{\emptyset}$ and a map $e \colon \cU \to M$ such that $\cU$ is a disjoint cover of $X$ and $\sum_{U \in \cU} 1_{X,R,U} e(U) = f$.
\end{itemize}
\end{prp}

\begin{proof}
The assertion (i) follows from assertion (ii). We show assertion (ii). By the finiteness of $f(X)$, $U_m \coloneqq f^{-1}(\ens{m})$ is a clopen subset of $X$ for any $m \in M$. Put $\cU \coloneqq \set{U_m}{m \in f(X)} \subset \CO(X) \setminus \ens{\emptyset}$. We denote by $e$ the map $\cU \to M$ assigning to each $U \in \cU$ the unique element of the singleton $f(U) \subset f(X)$. Then, $(\cU,e)$ satisfies the desired property.
\end{proof}

Let $X_0$ be a subset of $X$. We put $I_{X,R,X_0} \coloneqq \set{f \in \rCfin(X,R)}{f |_{X_0} = 0}$. We show a finite image analogue of \cite{BM21} Lemma 2.21.

\begin{prp}
\label{I is ind-projective}
Every finitely generated $\rCfin(X,R)$-submodule of $I_{X,R,X_0}$ is contained in a principal ideal generated by $1_{X,R,U}$ for some $U \in \CO(X)$ with $X_0 \cap U = \emptyset$.
\end{prp}

\begin{proof}
Since $\set{U \in \CO(X)}{X_0 \cap U = \emptyset}$ is directed by inclusions, it is reduced to the case of an ideal generated by a single element. Therefore it suffices to show that for any $f \in I_{X,R,X_0}$, there exists a $U \in \CO(X)$ such that $1_{X,R,U} f = f$ and $X_0 \cap U = \emptyset$. Put $U \coloneqq f^{-1}(R \setminus \ens{0}) \subset X$. We have $X_0 \cap U = \emptyset$ by $f |_{X_0} = 0$. By the finiteness of $f(X)$, we have $U \in \CO(X)$. For any $x \in U$, we have $(1_{X,R,U} f)(x) = 1_{X,R,U}(x) f(x) = 1 \cdot f(x) = f(x)$. For any $x \in X \setminus U$, we have $(1_{X,R,U} f)(x) = 1_{X,R,U}(x) f(x) = 0 = f(x)$. Therefore, we obtain $1_{X,R,U} f = f$.
\end{proof}

We abbreviate $(\cC,R)$ to $\cD_2$. We denote by $\Tpl$ the category of topological spaces and continuous maps, by $\CH \subset \Tpl$ the full subcategory of compact Hausdorff topological spaces, by $\TDCH \subset \CH$ the full subcategory of totally disconnected compact Hausdorff topological spaces, by $F_{\TDCH}^{\Tpl}$ the inclusion functor $\TDCH \hookrightarrow \Tpl$, by $\zeta$ the Banaschewski compactification functor $\Tpl \to \TDCH$ (cf.\ \cite{Mih14} Theorem 1.3), which is left adjoint to $F_{\TDCH}^{\Tpl}$, by $\iota_{\zeta,\bullet}$ the unit adjunction $\id_{\Tpl} \Rightarrow F_{\TDCH}^{\Tpl} \circ \zeta$, and by $\cM_{\cD_2} \colon \Alg_{\cC}(R) \to \CH$ the Berkovich spectrum functor (cf.\ \cite{Ber90} 1.2.1 Theorem). We recall that $\zeta(X)$ is explicitly described by using ultrafilters, and can be described as $\cM_{\cD_2}$ for the case where $R$ is a complete valued field by \cite{Mih14} Corollary 2.3. The correspondence assigning $\rCfin(X,R)$ to each $X \in \Tpl$ gives a functor $\Gamma_{\cD_2} \colon \Tpl^{\op} \to \Alg_{\cC}(R)$ by Proposition \ref{Cfin is closed}. We show an isolated analogue of \cite{BM21} Lemma 3.7 (ii).

\begin{prp}
\label{extension property}
If $M$ is isolated at $0$, then the pre-composition $\rC(\zeta(X),M) \to \rCfin(X,M)$ by $\iota_{\zeta,X} \colon X \to \zeta(X)$ is an isometric $R$-linear isomorphism.
\end{prp}

\begin{proof}
Since $\iota_{\zeta,X}(X)$ is dense in $\zeta(X)$ by \cite{Mih14} Lemma 1.8, the pre-composition by $\iota_{\zeta,X}$ is an isometry. Let $f \in \rCfin(X,R)$. Since $f(X)$ is a finite subset of the Hausdorff topological space $M$, $f(X)$ is a totally disconnected compact Hausdorff topological space. Therefore, by the universality of Banaschewski compactification, there exists a unique $\tl{f} \in \rC(\zeta(X),M)$ such that $\tl{f} \circ \iota_{\zeta,X} = f$. This implies the assertion.
\end{proof}

For each $U \in \CO(X)$, we denote by $\cl_{\zeta,X}(U) \subset \zeta(X)$ the closure of $\iota_{\zeta,X}(U)$. Proposition \ref{extension property} gives an alternative proof of the following well-known fact:

\begin{crl}
\label{extension property for clopen subset}
For any $U \in \CO(X)$, $\cl_{\zeta,X}(U)$ is the unique $\ol{U} \in \CO(\zeta(X))$ satisfying $\iota_{\zeta,X}^{-1}(\ol{U}) = U$.
\end{crl}

\begin{proof}
The assertion follows from Proposition \ref{extension property} applied to $R = M = (\F_2)_{\triv}$, because $\CO \colon \Tpl \to \Set$ is represented by $(\F_2)_{\triv}$.
\end{proof}

We abbreviate $(\cC,X,R)$ to $\cD_3$. For each $\cF \in \zeta(X)$, we denote by $\lim^{\cD_3}_{\cF}$ the $R$-algebra homomorphism $\rCfin(X,R) \to R$, which is bounded with respect to the restriction of $\n{\cdot}_{\rCbd(X,R)}$, assigning to each $f \in \rCfin(X,R)$ the limit $\lim_{\cF} f$, i.e.\ the unique $a \in R$ satisfying the property that there exists a $U \in \cF$ such that $f |_U$ is the constant map with value $a$. We denote by $F_{\Tpl}^{\Set}$ the forgetful functor $\Tpl \to \Set$. We show an isolated analogue of \cite{Mih14} Corollary 3.6 (iv), which gives a non-Archimedean counterpart of \Gelfand's transform.

\begin{prp}
\label{Gelfand transform}
If $R$ is isolated at $0$, then the correspondence assigning to each $Y \in \Tpl$ the map $\lim^{(\cC,Y,R)} \colon \zeta(Y) \to \Hom_{\Alg_{\cC}(R)}(\rCfin(Y,R),R), \ \cF \mapsto \lim^{(\cC,Y,R)}_{\cF}$ is a natural transformation $\lim^{(\cC,\bullet,R)} \colon F_{\Tpl}^{\Set} \circ \zeta \Rightarrow \Hom_{\Alg_{\cC}(R)}(\Gamma_{\cD_2},R)$, which consists of injective maps unless $R = \ens{0}$.
\end{prp}

\begin{proof}
The first assertion immediately follows from the construction of $\lim^{(\cC,Y,R)}$ for a $Y \in \Tpl$. We show the second assertion. By Proposition \ref{extension property}, it suffices to show that if $R \neq \ens{0}$, then for any $Y \in \TDCH$, the map $\ev \colon Y \to \Hom_{\Alg_{\cC}(R)}(\rC(Y,R),R),\ y \mapsto \ev_y$ is injective, where $\ev_y$ denotes the evaluation at $y$. Let $(y_0,y_1) \in Y^2$ with $y_0 \neq y_1$. Since $Y \in \TDCH$, there exists a $U \in \CO(Y)$ such that $y_0 \in U$ and $y_1 \notin U$. Then, we have $\ev_{y_0}(1_{X,R,U}) = 1 \neq 0 = \ev_{y_1}(1_{X,R,U})$ by $R \neq \ens{0}$. This implies $\ev_{y_0} \neq \ev_{y_1}$.
\end{proof}

As we announced at the beginning of this section, we will show \Gelfand's duality in Theorem \ref{Gelfand duality} under the assumption that $\cM(R)$ is connected, and this implies that $\lim^{(\cC,\bullet,R)}$ is a natural isomorphism under the same assumption. On the other hand, $\lim^{(\cC,\bullet,R)}$ is not a natural isomorphism if $R$ does not have precisely one non-zero idempotent element, as the following example shows:

\begin{exm}
\label{counterexample of Gelfand's duality}
Let $Y \in \TDCH$ denote the discrete set $\ens{0,1}$. If $R$ has no non-zero idempotent element, i.e.\ $R = \ens{0}$, then both of $\lim^{(\cC,Y,R)}(0)$ and $\lim^{(\cC,Y,R)}(1)$ are zero morphisms $\rCfin(Y,R) \to R$, and hence $\lim^{(\cC,Y,R)}$ is not injective. If $R$ admits a non-zero idempotent $e \neq 1$, then the map $\rCfin(Y,R) \to R, \ f \mapsto e f(0) + (1-e) f(1)$ is an element of $\Hom_{\Alg_{\cC}(R)}(\rCfin(Y,R),R)$ which does not belong to the image of $\lim^{(\cC,Y,R)}$, and hence $\lim^{(\cC,Y,R)}$ is not surjective.
\end{exm}

Suppose that $R$ is isolated at $0$. For each $x \in \cM_{\cD_2}(\rCfin(X,R))$, we denote by $\cF_{X,x}$ the subset $\set{U \in \CO(X)}{x(1_{X,R,U}) \neq 0}$ of $\CO(X)$, which is an ultrafilter of $\CO(X)$ by an argument completely parallel to that in \cite{Mih14} p.\ 10 for the case where $R$ is a complete valued field, and by $x|_R$ the composite of $x$ and the canonical ring homomorphism $R \to \rCfin(X,R)$. We denote by $G_{\cD_3}$ the map $\cM_{\cD_2}(\rCfin(X,R)) \to \zeta(X) \times \cM_{\cD_2}(R), \ x \mapsto (\cF_{X,x},x|_R)$. We show an isolated analogue of \cite{Mih14} Theorem 2.1.

\begin{thm}
\label{Banaschweski compatification}
If $R$ is isolated at $0$, then $G_{\cD_3}$ is a homeomorphism.
\end{thm}

\begin{proof}
An open basis of $\zeta(X) \times \cM_{\cD_2}(R)$ is given by the set of subsets of the form
\begin{eqnarray*}
\set{\cF \in \zeta(X)}{U \in \cF} \times \set{\ol{x} \in \cM_{\cD_2}(R)}{r_0 < \ol{x}(a) < r_1}
\end{eqnarray*}
with $(U,a,r_0,r_1) \in \CO(X) \times R \times (0,\infty)^2$, and its preimage by $G_{\cD_3}$ coincides with
\begin{eqnarray*}
\set{x \in \cM_{\cD_2}(\rCfin(X,R))}{x(1_{X,R,U}) \neq 0 \land r_0 < x(a) < r_1},
\end{eqnarray*}
which is an open subset of $\cM_{\cD_2}(\rCfin(X,R))$. This implies the continuity of $G_{\cD_3}$. Since $G_{\cD_3}$ is a continuous map between compact Hausdorff topological spaces, it suffices to show its bijectivity.

\vspace{0.1in}
First, let $(x_0,x_1) \in \cM_{\cD_2}(\rCfin(X,R))^2$ with $G_{\cD_3}(x_0) = G_{\cD_3}(x_1)$. In order to show $x_0 = x_1$, it suffices to show $x_0(f) = x_1(f)$ for any $f \in \rCfin(X,R)$. By Proposition \ref{idempotent generates C(X,R)} (ii), there exists a pair $(\cU,e)$ of a finite subset $\cU \subset \CO(X) \setminus \ens{\emptyset}$ and a map $e \colon \cU \to R$ such that $\cU$ is a disjoint cover of $X$ and $\sum_{U \in \cU} e(U) 1_{X,R,U} = f$. Put $\cF \coloneqq \cF_{X,x_0} = \cF_{X,x_1}$. If $\cU \cap \cF = \emptyset$, then the images of $1_{X,R,U}$ in the completed residue fields at $x_0$ and $x_1$ are $0$ for any $U \in \cF$, and hence so are the images of $f$. If $U \in \cF$ for some $U \in \cU$, then there is no other $U' \in \cU$ such that $U' \in \cF$ because $\cF$ is an ultrafilter, and hence we obtain
\begin{eqnarray*}
x_0(f) = x_0(e(U) 1_{X,R,U}) = x_0|_R(e(U)) = x_1|_R(e(U)) = x_1(e(U) 1_{X,R,U}) = x_1(f).
\end{eqnarray*}
Next, let $(\cF,\ol{x}) \in \zeta(X) \times \cM_{\cD_2}(R)$. We denote by $x \in \cM(\rCfin(X,R))$ the bounded multiplicative seminorm assigning $\ol{x}(\lim^{\cD_3}_{\cF} (f)) \in [0,\infty)$ to each $f \in \rCfin(X,R)$. We have
\begin{eqnarray*}
& & \cF_{X,x} = \set{U \in \CO(X)}{x(1_{X,R,U}) \neq 0} = \set{U \in \CO(X)}{\ol{x} \left( \textstyle\lim^{\cD_3}_{\cF}\displaystyle (1_{X,R,U}) \right) \neq 0} \\
& = & \set{U \in \CO(X)}{\textstyle\lim^{\cD_3}_{\cF}\displaystyle (1_{X,R,U}) = 1} = \set{U \in \CO(X)}{\exists V \in \cF, V \subset U} = \cF.
\end{eqnarray*}
For any $a \in R$, we have $x |_R(a) = x(a 1_{X,R,X}) = \ol{x}(\lim^{\cD_3}_{\cF} (a 1_{X,R,X})) = \ol{x}(a)$. This implies $x |_R = \ol{x}$. We obtain $G_{\cD_3}(x) = (\cF,\ol{x})$.
\end{proof}

By Theorem \ref{Banaschweski compatification}, we obtain an isolated analogue of \cite{Mih14} Corollary 2.3.

\begin{crl}
\label{Banaschweski compatification functor}
If $R$ is isolated at $0$, then the correspondence assigning $G_{(\cC,Y,R)}$ to each $Y \in \Tpl$ gives a natural isomorphism $G_{(\cC,\bullet,R)} \colon \cM_{\cD_2} \circ \Gamma_{\cD_2} \Rightarrow \zeta \times \cM_{\cD_2}(R)$
\end{crl}

We conclude this section by proving \Gelfand's duality for algebras of continuous functions over Banach rings isolated at $0$ with connected Berkovich spectrum.

\begin{thm} 
\label{Gelfand duality}
Let $R$ be a Banach ring isolated at $0$ such that $\cM(R)$ is connected. Then, the pair $(\zeta \circ \cM,\Gamma_{\cD_2})$ of functors realises a duality, i.e.\ a contravariant equivalence, between $\TDCH$ and the essential image of $\Gamma_{\cD_2}$.
\end{thm}

We note that in our context, a topological space is said to be {\it connected} if it has precisely one connected component, and hence $\emptyset$ is not connected.

\begin{proof}
We denote by $F_{\rL}$ the contravariant functor from the essential image of $\Gamma_{\cD_2}$ to $\TDCH$ induced by $\zeta \circ \cM$, and by $F_{\rR}$ the contravariant functor from $\TDCH$ to the essential image of $\Gamma_{\cD_2}$ induced by $\Gamma_{\cD_2}$. First, we construct a natural isomorphism $F_{\rL}(F_{\rR}(Y)) \to Y$ for a $Y \in \TDCH$. By Corollary \ref{Banaschweski compatification functor}, $G_{(\cC,Y,R)} \colon \cM(\Gamma_{\cD_2}(Y)) \to \zeta(Y) \times \cM(R)$ is an isomorphism in $\Tpl$. Since the first projection $\zeta(Y) \times \cM(R) \to \zeta(Y)$ is a quotient map whose fibres are connected, it satisfies the universality of the Banaschewski compactification. This implies that the natural morphisms
\begin{eqnarray*}
\xymatrix{
F_{\rL}(F_{\rR}(Y)) \ar@{=}[rr] \ar@{.>}[d] & & \zeta(\cM(\Gamma_{\cD_2}(Y))) \ar[d]^-{\zeta(G_{(\cC,Y,R)})} \\
Y \ar[r] & \zeta(Y) & \zeta(\zeta(Y) \times \cM(R)) \ar[l]
}
\end{eqnarray*}
in $\Tpl$ are isomorphisms, and hence there exists a unique dotted filler, which is a natural isomorphism in $\TDCH$.

\vspace{0.1in}
Next, we construct a natural isomorphism $\Gamma_{\cD_2}(Y) \to F_{\rR}(F_{\rL}(\Gamma_{\cD_2}(Y)))$ for a $Y \in \Tpl$. By Corollary \ref{Banaschweski compatification functor}, $G_{(\cC,Y,R)} \colon \cM(\Gamma_{\cD_2}(Y)) \to \zeta(Y) \times \cM(R)$ is an isomorphism in $\Tpl$. Since the first projection $\pi \colon \zeta(Y) \times \cM(R) \to \zeta(Y)$ is a quotient map whose fibres are homeomorphic to the connected topological space $\cM(R)$, it satisfies the universality of the Banaschewski compactification. This implies that the natural morphisms
\begin{eqnarray*}
\xymatrix{
\Gamma_{\cD_2}(Y) \ar@{.>}[d] & \Gamma_{\cD_2}(\zeta(Y)) \ar[r]^-{\Gamma_{\cD_2}(\pi)} \ar[l] & \Gamma_{\cD_2}(\zeta(Y) \times \cM(R)) \ar[d]^-{\Gamma_{\cD_2}(G_{(\cC,Y,R)})} \\
F_{\rR}(F_{\rL}(\Gamma_{\cD_2}(Y))) \ar@{=}[r] & \Gamma_{\cD_2}(\zeta(\cM(\Gamma_{\cD_2}(Y)))) \ar[r] & \Gamma_{\cD_2}(\cM(\Gamma_{\cD_2}(Y))) \\
}
\end{eqnarray*}
in $\Alg_{\cC}(R)$ are isomorphisms by Proposition \ref{extension property}, and there exists a unique dotted filler, which is a natural isomorphism in $\Alg_{\cC}(R)$.
\end{proof}

\begin{rmk}
In the case where $R = (\Z, \v{\cdot}_\infty)$, and under the hypothesis that $X$ is the Stone--\Cech compactification of a discrete topological space, Theorem \ref{Gelfand duality} has been recently independently obtained by Kobi Kremnizer and Andr\'e Henriques (private communication with the first listed author).
\end{rmk}

We say that $R$ {\it satisfies Shilov idempotent theorem} if every $U \in \CO(\cM(R))$ coincides with $\set{x \in \cM(R)}{x(e) = 1}$ for some idempotent $e \in R$. By the definition, Shilov idempotent theorem is a weak variant of Tate's acyclicity focusing only on a disjoint clopen covering. For example, every Banach algebra over a complete valued field, e.g.\ a trivially valued field, satisfies Shilov idempotent theorem by \cite{Ber90} 7.4.1 Theorem.

\begin{crl}
Let $R$ be a Banach ring isolated at $0$ satisfying Shilov idempotent theorem. Then, the pair $(\zeta \circ \cM,\Gamma_{\cD_2})$ of functors realises a duality between $\TDCH$ and the essential image of $\Gamma_{\cD_2}$ if and only if $R$ has precisely one non-zero idempotent element.
\end{crl}

\begin{proof}
We follow the convention in the proof of Theorem \ref{Gelfand duality}. If $R$ does not have precisely one non-zero idempotent element, then $F_{\rR}$ is not fully faithful, because Example \ref{counterexample of Gelfand's duality} implies that $F_{\rR}(\pi)$ is not bijective, where $\pi$ denotes the unique continuous map $\ens{0,1} \to \ens{0}$, where $\ens{0,1}$ is equipped with the discrete topology. If $R$ has precisely one non-zero idempotent element, then $\cM(R)$ is connected by Shilov idempotent theorem, and hence the assertion follows from Theorem \ref{Gelfand duality}. 
\end{proof}

\section{Homotopy epimorphisms for functions valued in discrete Banach rings}
\label{Homotopy epimorphisms for functions valued in discrete Banach rings}

By Theorem \ref{Gelfand duality}, which is \Gelfand's duality over Banach rings isolated at $0$, morphisms $\rC(X, R) \to \rC(Y, R)$ of Banach $R$-algebras correspond to continuous maps $Y \to X$ between totally disconnected compact Hausdorff topological spaces . In this section, we characterise the morphisms between such algebras corresponding to closed topological embeddings as homotopy epimorphisms. The notion of homotopy epimorphism is the most basic notion of localisation of (simplicial) rings used in derived geometry. We refer the reader to \cite{BM21} \S 1.3 for a more detailed discussion of this notion and its use in the context of the theory of Banach rings and in analytic geometry.

\vspace{0.1in}
For an additive category $\mathcal{E}$, we denote by $\Chba(\mathcal{E})$ the additive category of chain complexes bounded above associated to $\mathcal{E}$ and chain morphisms. We put $\Chba_{\cC}(R) \coloneqq \Chba(\Mod_{\cC}(R))$. We denote by $\Derba_{\cC}(R)$ the derived category of $\Chba_{\cC}(R)$, i.e.\ the localisation of the homotopy category of $\Chba_{\cC}(R)$ by the multiplicative system of morphisms represented by quasi-isomorphisms (cf.\ \cite{BM21} \S 1.2 for more details about this construction). We recall that the completed tensor product $(\cdot) \wh{\otimes}_R (\cdot) \colon \Mod_{\cC}(R)^2 \to \Mod_{\cC}(R)$ can be derived to a functor $(\cdot) \wh{\otimes}_R^\bL (\cdot) \colon \Derba_{\cC}(R)^2\to \Derba_{\cC}(R)$ in a way similar to the algebraic tensor product. A morphism $\pi \colon A_0 \to A_1$ in $\Alg_{\cC}(R)$ is said to be a {\it homotopy epimorphism in $\Alg_{\cC}(R)$} if the multiplication $A_1 \wh{\otimes}_{A_0}^{\bL} A_1 \to A_1$ is an isomorphism in $\Derba_{\cC}(A_0)$.

\vspace{0.1in}
Suppose that $R$ is isolated at $0$. Let $(K,j) \in \Tpl/X$, where $\Tpl/X$ denotes the slice category of $\Tpl$ over $X$. We regard $\rCfin(K,R)$ as an object of $\Alg_{\cC}(\rCfin(X,R))$ through $\Gamma_{\cD_2}(j)$. We say that $j$ is a $\zeta$-embedding if it is an $\NBAb$-embedding in the sense of \cite{BM21} \S 3.1, i.e.\ if $\zeta(j) \colon \zeta(K) \to \zeta(X)$ is a homeomorphism onto the image. We give an isolated analogue of \cite{BM21} Lemma 3.9.

\begin{thm}
\label{embedding corresponds to homotopy epimorphism}
If $R$ is isolated at $0$, then the following are equivalent:
\begin{itemize}
\item[(i)] The morphism $\Gamma_{\cD_2}(j)$ is a homotopy epimorphism in $\Alg_{\cC}(R)$.
\item[(ii)] The morphism $\Gamma_{\cD_2}(j)$ is a strict epimorphism in $\Mod_{\cC}(R)$.
\item[(iii)] The map $j$ is a $\zeta$-embedding or $R = \ens{0}$.
\end{itemize}
\end{thm}

In order to verify Theorem \ref{embedding corresponds to homotopy epimorphism}, we show several lemmata. First, we show an isolated analogue of \cite{BM21} Lemma 2.9 and \cite{BM21} Lemma 2.22.

\begin{lmm}
\label{ker(j) is ind-projective}
Every finitely generated $\rCfin(X,R)$-submodule of $\set{f \in \rCfin(X,R)}{f \circ j = 0}$ is contained in a principal ideal generated by $1_{X,R,U}$ for some $U \in \CO(X)$ with $j(K) \cap U = \emptyset$.
\end{lmm}

\begin{proof}
The assertion follows from Proposition \ref{I is ind-projective} applied to the subset $j(K)$ and the corresponding ideal $\set{f \in \rCfin(X,R)}{f \circ j = 0} = \set{f \in \rCfin(X,R)}{\forall x \in K, f(j(x)) = 0} = I_{X,R,j(K)}$.
\end{proof}

We denote by $\cC_{\leq 1} \subset \cC$ the symmetric monoidal subcategory with the same class of objects and the class of morphisms consisting of morphisms of operator norm $\leq 1$. We show an isolated analogue of \cite{BM21} Lemma 2.13 and \cite{BM21} Lemma 2.24.

\begin{lmm}
\label{Tietze}
Suppose that $R$ is isolated at $0$ and $j$ is a $\zeta$-embedding.
\begin{itemize}
\item[(i)] For any $f \in \rCfin(K,R)$, there exists an $\tl{f} \in \rCfin(X,R)$ such that $\tl{f} \circ j = f$ and $\n{\tl{f}}_{\rCfin(X,R)} = \n{f}_{\rCfin(K,R)}$.
\item[(ii)] The morphism $\Gamma_{\cD_2}(j)$ is a strict epimorphism in $\Mod_{\cC}(\rCfin(X,R))$.
\item[(iii)] If $R$ is a monoid object of $\cC_{\leq 1}$, then the morphism $\Gamma_{\cD_2}(j)$ is a strict epimorphism in $\Mod_{\cC_{\leq 1}}(\rCfin(X,R))$.
\end{itemize}
\end{lmm}

\begin{proof}
The assertions (ii) and (iii) immediately follow from the assertion (i). We show the assertion (i). Take an $\epsilon \in (0,\infty)$ with $\ens{0} = \set{a \in R}{\n{a}_R < \epsilon}$. By Proposition \ref{extension property}, there exists a unique $g \in \rC(\zeta(K),R)$ such that $g \circ \iota_{\zeta,K} = f$ and $\n{g}_{\rC(\zeta(K),R)} = \n{f}_{\rCfin(K,R)}$. Since $j$ is a $\zeta$-embedding, there exists a $\tl{g} \in \rC(\zeta(X),R)$ such that $\tl{g} \circ \zeta(j) = g$ and $\n{\tl{g}}_{\rC(\zeta(X),R)} \leq \n{g}_{\rC(\zeta(K),R)} + \epsilon$ by \cite{BM21} Lemma 2.13. By the strictness of $\Gamma_{\cD_2}(\zeta(j))$ and the arbitrariness of $\epsilon$, we obtain $\n{\tl{g}}_{\rC(\zeta(X),R)} = \n{g}_{\rC(\zeta(K),R)}$.

\vspace{0.1in}
Put $\tl{f} \coloneqq \tl{g} \circ \iota_{\zeta,X} \in \rCfin(X,R)$. For any $x \in K$, we have $\tl{f}(j(x)) = \tl{g}(\iota_{\zeta,X}(j(x))) = \tl{g}(\zeta(j)(\iota_{\zeta,K}(x))) = g(\iota_{\zeta,K}(x)) = f(x)$ because $\iota_{\zeta,\bullet}$ is a natural transform $\id_{\CH} \Rightarrow \TDCH$. This implies $\tl{f} \circ j = f$ and $\n{\tl{f}}_{\rCfin(X,R)} \geq \n{f}_{\rCfin(K,R)}$. By $\n{\tl{f}}_{\rCfin(X,R)} \leq \n{\tl{g}}_{\rC(\zeta(X),R)} = \n{g}_{\rC(\zeta(K),R)} = \n{f}_{\rCfin(K,R)}$, we obtain $\n{\tl{f}}_{\rCfin(X,R)} = \n{f}_{\rCfin(K,R)}$.
\end{proof}

From now on and until the end of this section, we use the notation $I \coloneqq \ker(\Gamma_{\cD_2}(j))$ and $i \colon I \hookrightarrow \rCfin(X,R)$ for the canonical inclusion.

\begin{lmm}
\label{approximation by idempotent}
Suppose that $R$ is isolated at $0$. Let $M \in \Mod_{\cC}(\rCfin(X,R))$. For any $(\tl{m},\epsilon) \in (I \wh{\otimes}_{\rCfin(X,R)} M) \times (0,\infty)$, there exists a $(U,m) \in \CO(X) \times M$ such that $j(K) \cap U = \emptyset$ and $\n{\tl{m} - 1_{X,R,U} \otimes m}_{I \wh{\otimes}_{\rCfin(X,R)} M} < \epsilon$.
\end{lmm}

\begin{proof}
By the definition of $\wh{\otimes}_{\rCfin(X,R)}$, there exists a tuple $(k,f,n)$ of a $k \in \N$, an $f = (f_h)_{h=0}^{k} \in I^{k+1}$, and an $n = (n_h)_{h=0}^{k} \in M^{k+1}$ such that $\n{\tl{m} - \sum_{h=0}^{k} f_h \otimes n_h}_{I \wh{\otimes}_{\rCfin(X,R)} M} < \epsilon$. By Lemma \ref{ker(j) is ind-projective}, there exists a $U \in \CO(X)$ such that $1_{X,R,U} f_h = f_h$ and $j(K) \cap U = \emptyset$ for any $0 \leq h \leq k$. Put $m \coloneqq \sum_{h=0}^{k} f_h n_h \in M$. We have
\begin{eqnarray*}
\sum_{h=0}^{k} f_h \otimes n_h = \sum_{h=0}^{k} 1_{X,R,U} f_h \otimes n_h = \sum_{h=0}^{k} 1_{X,R,U} \otimes f_h n_h = 1_{X,R,U} \otimes \sum_{h=0}^{k} f_h n_h = 1_{X,R,U} \otimes m
\end{eqnarray*}
in $I \otimes_{\rCfin(X,R)} M$, and hence $\n{\tl{m} - 1_{X,R,U} \otimes m}_{I \wh{\otimes}_{\rCfin(X,R)} M} < \epsilon$.
\end{proof}

\begin{lmm}
\label{I is stably strict monomorphism}
Suppose that $R$ is isolated at $0$. Let $M \in \Mod_{\cC}(\rCfin(X,R))$. Then, the morphism $i_M \colon I \wh{\otimes}_{\rCfin(X,R)} M \to M$ in $\Mod_{\cC}(\rCfin(X,R))$ induced by the inclusion $i \colon I \hookrightarrow \rCfin(X,R)$ is a strict monomorphism.
\end{lmm}

\begin{proof}
We abbreviate $I \wh{\otimes}_{\rCfin(X,R)} M$ to $M_I$. It suffices to show $\n{\tl{m}}_{M_I} \leq \n{1}_R \n{i_M(\tl{m})}_M$ for any $\tl{m} \in M_I$. Assume $\n{\tl{m}}_{M_I} > \n{1}_R \n{i_M(\tl{m})}_M$. Put $\epsilon \coloneqq 2^{-1}(\n{\tl{m}}_{M_I} - \n{1}_R \n{i_M(\tl{m})}_M) \in (0,\infty)$. We denote by $C_{i_M}$ the operator norm of $i_M$. By Lemma \ref{approximation by idempotent}, there exists a $(U,m) \in \CO(X) \times M$ such that $\n{\tl{m} - 1_{X,R,U} \otimes m}_{M_I} < (\n{1}_R C_{i_M} + 1)^{-1} \epsilon$ and $j(K) \cap U = \emptyset$. Put $\delta \coloneqq \tl{m} - 1_{X,R,U} \otimes m \in M_I$. We have
\begin{eqnarray*}
& & \n{\tl{m}}_{M_I} = \n{1_{X,R,U} \otimes m + \delta}_{M_I} = \n{1_{X,R,U} \otimes 1_{X,R,U} m + \delta}_{M_I} \\
& \leq & \n{1_{X,R,U} \otimes 1_{X,R,U} m}_{M_I} + \n{\delta}_{M_I} \leq \n{1_{X,R,U}}_I \n{1_{X,R,U} m}_M + \n{\delta}_{M_I} \\
& \leq & \n{1}_R \n{1_{X,R,U} m}_M + \n{\delta}_{M_I} = \n{1}_R \n{i_M(1_{X,R,U} \otimes m)}_M + \n{\delta}_{M_I} \\
& = & \n{1}_R \n{i_M(\tl{m} - \delta)}_M + \n{\delta}_{M_I} \leq \n{1}_R (\n{i_M(\tl{m})}_M + \n{i_M(\delta)}_M) + \n{\delta}_{M_I} \\
& \leq & \n{1}_R \n{i_M(\tl{m})}_M + (\n{1}_R C_{i_M} + 1) \n{\delta}_{M_I} < \n{1}_R \n{i_M(\tl{m})}_M + \epsilon,
\end{eqnarray*}
but this contradicts the definition of $\epsilon$.
\end{proof}

We show an isolated analogue of \cite{BM21} Lemma 2.22.

\begin{lmm}
\label{flatness of I}
If $R$ is isolated at $0$, then the objects $I$ and $\rCfin(X,R)$ of $\Mod_{\cC}(\rCfin(X,R))$ are strongly flat.
\end{lmm}

\begin{proof}
The strong flatness of $\rCfin(X,R)$ is obvious because it is a monoidal unit with respect to $\wh{\otimes}_{\rCfin(X,R)}$. Let $E$ be a strictly exact sequence
\begin{eqnarray*}
0 \to M \stackrel{s}{\to} N \stackrel{t}{\to} L
\end{eqnarray*}
of $\Mod_{\cC}(\rCfin(X,R))$. Following the convention in Lemma \ref{I is stably strict monomorphism} and its proof, we put $M_I \coloneqq I \wh{\otimes}_{\rCfin(X,R)} M$, $N_I \coloneqq I \wh{\otimes}_{\rCfin(X,R)} N$, and $L_I \coloneqq I \wh{\otimes}_{\rCfin(X,R)} L$, and denote by $i_M \colon M_I \to M$, $i_N \colon N_I \to N$, $i_L \colon L_I \to L$ the morphisms induced by the inclusions. We consider the commutative diagram
\begin{eqnarray*}
\xymatrix@C4em{
0 \ar[r] \ar@{=}[d] & M_I \ar[r]^-{I \wh{\otimes}_{\rCfin(X,R)} s} \ar[d]^-{i_M} & N_I \ar[r]^-{I \wh{\otimes}_{\rCfin(X,R)} t} \ar[d]^-{i_N} & L_I \ar[d]^-{i_L} \\
0 \ar[r] & M \ar[r]^-{s} & N \ar[r]^-{t} & L \\
}
\end{eqnarray*}
of $\Mod_{\cC}(\rCfin(X,R))$. By Lemma \ref{I is stably strict monomorphism}, $i_M$ and $i_N$ are strict monomorphisms. By the strict exactness of $E$, $s$ is a strict monomorphism. By the commutativity of the second square, $I \wh{\otimes}_{\rCfin(X,R)} s$ is a strict monomorphism. Therefore, it suffices to show that for any $(\tl{n},\epsilon) \in \ker(I \wh{\otimes}_{\rCfin(X,R)} t) \times (0,\infty)$, there exists an $\tl{m} \in M_I$ such that $\n{\tl{n} - (I \wh{\otimes}_{\rCfin(X,R)} s)(\tl{m})}_{N_I} < \n{1}_R^{-1} \epsilon$.

\vspace{0.1in}
We denote by $C_N$ the operator norm of the scalar multiplication $\rCfin(X,R) \wh{\otimes} N \to N$, by $C_{N_I}$ the operator norm of the scalar multiplication $\rCfin(X,R) \wh{\otimes} (N_I) \to N_I$, and by $D_{i_N}$ the operator norm of the inverse of the isomorphism
\begin{eqnarray*}
N_I \to \im(i_N)
\end{eqnarray*}
in $\Mod_{\cC}(\rCfin(X,R))$ induced by the strict monomorphism $i_N$. By Lemma \ref{approximation by idempotent}, there exists a $(U,n) \in \CO(X) \times N$ such that $\n{\tl{n} - 1_{X,R,U} \otimes n}_{N_I} < 2^{-1} C_{N_I}^{-1} \n{1}_R^{-1} \epsilon$ and $j(K) \cap U = \emptyset$. By the exactness of $E$ and the commutativity of the diagram, there exists an $m \in M$ such that $\n{i_N(\tl{n}) - s(m)}_N < 2^{-1} D_{i_N}^{-1} C_N^{-1} \n{1}_R^{-1} \epsilon$. Put $\tl{m} \coloneqq 1_{X,R,U} \otimes m \in M_I$. We obtain
\begin{eqnarray*}
& & \n{\tl{n} - (I \wh{\otimes}_{\rCfin(X,R)} s)(\tl{m})}_{N_I} = \n{\tl{n} - 1_{X,R,U} \otimes s(m)}_{N_I} \\
& \leq & \n{(1-1_{X,R,U})\tl{n}}_{N_I} + \n{1_{X,R,U} \tl{n} - 1_{X,R,U} \otimes s(m)}_{N_I} \\
& = & \n{(1-1_{X,R,U})(\tl{n} - 1_{X,R,U} \otimes n)}_{N_I} + \n{1_{X,R,U} \tl{n} - 1_{X,R,U} \otimes s(m)}_{N_I} \\
& \leq & \n{(1-1_{X,R,U})(\tl{n} - 1_{X,R,U} \otimes n)}_{N_I} \\
& & + D_{i_N} \n{(i_N)(1_{X,R,U} \tl{n} - 1_{X,R,U} \otimes s(m))}_N \\
& = & \n{(1-1_{X,R,U})(\tl{n} - 1_{X,R,U} \otimes n)}_{N_I} \\
& & + D_{i_N} \n{1_{X,R,U}((i_N)(\tl{n}) - s(m))}_N \\
& \leq & C_{N_I} \n{1-1_{X,R,U}}_{\rCfin(X,R)} \n{\tl{n} - 1_{X,R,U} \otimes n}_{N_I} \\
& & + D_{i_N} C_N \n{1_{X,R,U}}_{\rCfin(X,R)} \n{(i_N)(\tl{n}) - s(m)}_N \\
& \leq & C_{N_I} \n{1}_R \n{\tl{n} - 1_{X,R,U} \otimes n}_{N_I} + D_{i_N} C_N \n{1}_R \n{i_N(\tl{n}) - s(m)}_N.
\end{eqnarray*}
By the choice of $(U,n)$ and $m$, the right hand side is smaller than $2^{-1} \epsilon + 2^{-1} \epsilon = \epsilon$.
\end{proof}

We denote by $\cP_{\cC}^R$ the projective resolution functor $\Chba_{\cC}(R) \to \Chba_{\cC}(R)$ explicitly introduced in \cite{BM21} at the beginning of p.\ 13. In the next lemma, we perform a basic computation of the derived tensor in a way completely parallel to the algebraic setting.

\begin{lmm}
\label{computation of derived tensor}
Let $A$ be a monoid object of $\cC$. For any pair of morphisms $p_0 \colon P_0 \to M_0$ and $p_1 \colon P_1 \to M_1$ in $\Chba_{\cC}(A)$, if either $P_0$ or $P_1$ is termwise strongly flat and both of $p_0$ and $p_1$ are quasi-isomorphisms, then the span
\begin{eqnarray*}
\xymatrix{
\Tot(\cP_{\cC}^A(P_0) \wh{\otimes}_A \cP_{\cC}^A(P_1)) \ar[r] \ar[d]^-{\Tot(\cP_{\cC}^A(p_0) \wh{\otimes}_A \cP_{\cC}^A(p_1))} & \Tot(P_0 \wh{\otimes}_A P_1) \\
\Tot(\cP_{\cC}^A(M_0) \wh{\otimes}_A \cP_{\cC}^A(M_1)) &
}
\end{eqnarray*}
in $\Chba_{\cC}(A)$ represents both an isomorphism $M_0 \wh{\otimes}_A^{\bL} M_1 \to \Tot(P_0 \wh{\otimes}_A P_1)$ in $\Derba_{\cC}(A)$ and its inverse.
\end{lmm}

\begin{proof}
By \cite{BM21} Proposition 1.6 (i), we may assume that $P_1$ is termwise strongly flat. It suffices to show that every arrow in the commutative diagram
\begin{eqnarray*}
\xymatrix{
\Tot(\cP_{\cC}^A(P_0) \wh{\otimes}_A \cP_{\cC}^A(P_1)) \ar[r] \ar[d]^-{\Tot(\cP_{\cC}^A(P_0) \wh{\otimes}_A \cP_{\cC}^A(p_1))} & \Tot(P_0 \wh{\otimes}_A P_1) \ar[d]^-{\Tot(p_0 \wh{\otimes}_A P_1)} \\
\Tot(\cP_{\cC}^A(M_0) \wh{\otimes}_A \cP_{\cC}^A(P_1)) \ar[r] \ar[d]^-{\Tot(\cP_{\cC}^A(p_0) \wh{\otimes}_A \cP_{\cC}^A(M_1))} & \Tot(M_0 \wh{\otimes}_A P_1) \\
\Tot(\cP_{\cC}^A(M_0) \wh{\otimes}_A \cP_{\cC}^A(M_1)) &
}
\end{eqnarray*}
of $\Chba_{\cC}(A)$ represents an isomorphism in $\Derba_{\cC}(A)$. Since $P_1$ is termwise strongly flat, the horizontal arrows are quasi-isomorphisms by \cite{BM21} Proposition 1.6 (iii). Since $P_1$ is termwise strongly flat and $p_0$ is a quasi-isomorphism, the right vertical arrow is a quasi-isomorphism by \cite{BM21} Proposition 1.6 (iv). By the commutativity of the top square, the top left vertical arrow represents an isomorphism in $\Derba_{\cC}(A)$. In order to show that the bottom left vertical arrow $\Tot(\cP_{\cC}^A(p_0) \wh{\otimes}_A \cP_{\cC}^A(M_1))$ represents an isomorphism in $\Derba_{\cC}(A)$, we consider the commutative diagram
\begin{eqnarray*}
\xymatrix{
\cP_{\cC}^A(P_1) \ar[r] \ar[d]^-{\cP_{\cC}^A(p_1)} & P_1 \ar[d]^-{p_1}\\
\cP_{\cC}^A(M_1) \ar[r] & M_1
}
\end{eqnarray*}
of $\Chba_{\cC}(A)$. By the definition of the projective resolution functor $\cP_{\cC}^A$, the two horizontal arrows are quasi-isomorphisms. Since $p_1$ is a quasi-isomorphism, $\cP_{\cC}^A(p_1)$ represents an isomorphism in $\Derba_{\cC}(A)$ by the commutativity of the diagram. This implies that $\cP_{\cC}^A(p_1)$ is a quasi-isomorphism by \cite{BM21} Corollary 1.5 (i). Since $\cP_{\cC}^A(M_0)$ is termwise projective and $\cP_{\cC}^A(p_0)$ is a quasi-isomorphism, the bottom left vertical arrow $\Tot(P_{\cC}^A(p_0) \wh{\otimes}_A P_{\cC}^A(M_1))$ is a quasi-isomorphism by \cite{BK20} Proposition 3.11 and \cite{BM21} Proposition 1.6 (iv).
\end{proof}

\begin{lmm}
\label{stably exactness}
Let $M \in \Mod_{\cC}(\rCfin(X,R))$. If $R$ is isolated at $0$ and $j$ is a $\zeta$-embedding, then the chain complex
\begin{eqnarray*}
0 \to I \wh{\otimes}_{\rCfin(X,R)} M \to M \to \rCfin(K,R) \wh{\otimes}_{\rCfin(X,R)} M \to 0
\end{eqnarray*}
of $\Mod_{\cC}(\rCfin(X,R))$ is strictly exact.
\end{lmm}

\begin{proof}
The assertion immediately follows from the strong right exactness of $(\cdot) \wh{\otimes}_{\rCfin(X,R)} M$ and Lemma \ref{I is stably strict monomorphism}.
\end{proof}

\begin{lmm}
\label{derived tensor for embedding is tensor}
Let $M \in \Mod_{\cC}(\rCfin(X,R))$. If $R$ is isolated at $0$ and $j$ is a $\zeta$-embedding, then the natural morphism $\rCfin(K,R) \wh{\otimes}_{\rCfin(X,R)}^{\bL} M \to \rCfin(K,R) \wh{\otimes}_{\rCfin(X,R)} M$ in $\Der_{\cC}(\rCfin(X,R))$ is an isomorphism.
\end{lmm}

\begin{proof}
We denote by $E \in \Chba_{\cC}(\rCfin(X,R))$ the chain complex
\begin{eqnarray*}
\cdots \to 0 \to I \to \rCfin(X,R) \to 0 \to \cdots
\end{eqnarray*}
of $\Mod_{\cC}(\rCfin(X,R))$, and by $p$ the morphism $E \to \rCfin(K,R)$ in $\Chba_{\cC}(\rCfin(X,R))$ associated to $\Gamma_{\cD_2}(j)$. By Lemma \ref{stably exactness}, $p \wh{\otimes}_{\rCfin(X,R)} M$ is a quasi-isomorphism. Moreover, $E$ is termwise strongly flat by Lemma \ref{flatness of I}. Therefore, the assertion follows from Lemma \ref{computation of derived tensor} applied to $(p,\id_M)$.
\end{proof}

We record the following consequence of Lemma \ref{derived tensor for embedding is tensor} for later use.

\begin{crl}
\label{acyclic implies flat}
The object $\rCfin(K,R)$ of $\Mod_{\cC}(\rCfin(X,R))$ is flat, in the sense that the functor $(\cdot) \wh{\otimes}_{\rCfin(X,R)} \rCfin(K,R) \colon \Mod_{\cC}(\rCfin(X,R)) \to \Mod_{\cC}(\rCfin(K,R))$ is exact.
\end{crl}

\begin{proof}
Let $E$ be a strictly exact sequence $0 \to M \to N \to L \to 0$ of $\Mod_{\cC}(\rCfin(X,R))$ regarded as an object of $\Derba_{\cC}(\rCfin(X,R))$. Since $E$ is a zero object of $\Derba_{\cC}(\rCfin(X,R))$ and $(\cdot) \wh{\otimes}_{\rCfin(X,R)}^{\bL} \rCfin(K,R)$ is a functor, $E \wh{\otimes}_{\rCfin(X,R)}^{\bL} \rCfin(K,R)$ is a zero object of $\Derba_{\cC}(\rCfin(K,R))$. This implies that $E \wh{\otimes}_{\rCfin(X,R)} \rCfin(K,R)$ is a zero object of $\Derba_{\cC}(\rCfin(K,R))$ by Lemma \ref{derived tensor for embedding is tensor}, and hence is strictly exact by \cite{BM21} Corollary 1.5 (iii).
\end{proof}

Finally, we prove Theorem \ref{embedding corresponds to homotopy epimorphism} in a way completely parallel to the proof of \cite{BM21} Lemma 3.9.

\begin{proof}[Proof of Theorem \ref{embedding corresponds to homotopy epimorphism}]
The condition (iii) implies the condition (ii) by Lemma \ref{Tietze} (ii). The condition (ii) implies the condition (i) by Lemma \ref{derived tensor for embedding is tensor} applied to $M = \rCfin(K,R)$ and \cite{BM21} Proposition 1.8. Suppose that the condition (i) holds and $R \neq \ens{0}$. Then, $\Gamma_{\cD_2}(j)$ is an epimorphism in $\Alg_{\cC}(R)$ by \cite{BM21} Proposition 1.7, and hence the map
\begin{eqnarray*}
\Hom_{\Alg_{\cC}(R)}(\rCfin(K,R),R) & \to & \Hom_{\Alg_{\cC}(R)}(\rCfin(X,R),R) \\
\phi & \mapsto & \phi \circ \Gamma_{\cD_2}(j)
\end{eqnarray*}
is injective. This implies the injectivity of $\zeta(j)$ by Proposition \ref{Gelfand transform}. Since $\zeta(K)$ and $\zeta(X)$ are compact Hausdorff topological spaces, $\zeta(j)$ is a homeomorphism onto the image.
\end{proof}

\section{Tate Acyclicity for functions valued in discrete Banach rings}
\label{Tate Acyclicity for functions valued in discrete Banach rings}

The next step after the algebraic characterisation of closed topological embeddings is the characterisation of topological covers. We give such a characterisation in this section by showing that a finite set of closed topological embeddings between totally disconnected compact Hausdorff topological spaces is a (set-theoretic) cover if and only if the associated derived (or equivalently, non-derived) Tate--\Cech complex is strictly exact. This result, combined with the characterisation of closed topological embeddings as homotopy epimorphisms given in \S \ref{Homotopy epimorphisms for functions valued in discrete Banach rings}, shows that the topology of totally disconnected compact Hausdorff topological spaces agree with the restriction of the formal homotopy Zariski topology on the dual category of $\Alg_\cC(R)$ when totally disconnected compact Hausdorff topological spaces are considered as objects of the dual category of algebras of continuous functions to $R$ through \Gelfand's duality (Theorem \ref{Gelfand duality}). We refer the reader to \cite{BM21} \S 1.3 for more details on this.

\vspace{0.1in}
For an $A \in \Alg_{\cC}(R)$, a {\it non-derived} (resp.\ {\it derived}) {\it cover} of $A$ in $\cC$ is an $S \subset \Alg_{\cC}(A)$ satisfying the following:
\begin{itemize}
\item[(i)] For any $B \in S$, the structure morphism $A \to B$ is an epimorphism (resp.\ a homotopy epimorphism) in $\Alg_{\cC}(R)$.
\item[(ii)] There exists a finite subset $S_0 \subset S$ satisfying Tate's acyclicity (resp.\ derived Tate's acyclicity) (cf.\ \cite{BK20} Theorem 2.15 and \cite{BM21} pp.\ 15--16).
\end{itemize}
Suppose that $R$ is isolated at $0$. We keep the abbreviations $(\cC,R) \coloneqq \cD_2$ and $(\cC,X,R) \coloneqq \cD_3$ introduced before. We denote by $\Gamma_{\cD_3}$ the functor $(\Tpl/X)^{\op} \to \Alg_{\cC}(\rCfin(X,R))$ induced by $\Gamma_{\cD_2}$. For an $S \subset \Tpl/X$, we put $\Gamma^{\cD_3}_*(S) \coloneqq \set{\Gamma_{\cD_3}(K,j)}{(K,j) \in S}$. A {\it $\zeta$-cover} of $X$ is an $S \subset \Tpl/X$ satisfying the following:
\begin{itemize}
\item[(i)] For any $(K,j) \in S$, $j$ is a $\zeta$-embedding.
\item[(ii)] There exists a finite subset $S_0 \subset S$ such that $\bigcup_{(K,j) \in S_0} \zeta(j)(\zeta(K)) = \zeta(X)$.
\end{itemize}
We give an isolated analogue of \cite{BM21} Theorem 3.1, which is itself an analogue of Tate's acyclicity in rigid geometry.

\begin{thm}
\label{Tate's acyclicity for closed immersion}
Let $S \subset \Tpl/X$. If $R$ is isolated at $0$, $X \in \TDCH$, and $S$ consists of closed immersions, then the following are equivalent:
\begin{itemize}
\item[(i)] The set $S$ is a $\zeta$-cover of $X$ or $R = \ens{0}$.
\item[(ii)] The set $S$ admits a finite subset $S_0$ with $\bigcup_{(K,j) \in S} j(K) = X$ or $R = \ens{0}$.
\item[(iii)] The set $\Gamma^{\cD_3}_*(S)$ is a derived cover of $\rC(X,R)$ in $\cC$.
\item[(iv)] The set $\Gamma^{\cD_3}_*(S)$ is a non-derived cover of $\rC(X,R)$ in $\cC$.
\end{itemize}
\end{thm}

In order to show Theorem \ref{Tate's acyclicity for closed immersion}, we prepare a couple of lemmata. First, we show an isolated analogue of \cite{BM21} Lemma 3.2. Recall that for any $X_0 \subset X$ we put $I_{X,R,X_0} \coloneqq \set{f \in \rCfin(X,R)}{f |_{X_0} = 0}$.

\begin{lmm}
\label{subset operation and aritmetic operation}
Let $K_0$ and $K_1$ be subsets of $X$.
\begin{itemize}
\item[(i)] The equalities $I_{X,R,K_0} I_{X,R,K_1} = I_{X,R,K_0 \cup K_1} = I_{X,R,K_0} \cap I_{X,R,K_1}$ hold.
\item[(ii)] If $K_0$ and $K_1$ are closed, $X \in \TDCH$, and $R$ is isolated at $0$, then the addition $I_{X,R,K_0} \oplus I_{X,R,K_1} \to I_{X,R,K_0 \cap K_1}$ is a strict epimorphism in $\Mod_{\cC}(\rC(X,R))$.
\end{itemize}
\end{lmm}

\begin{proof}
The relations $I_{X,R,K_0 \cup K_1} = I_{X,R,K_0} \cap I_{X,R,K_1}$ and $I_{X,R,K_0} I_{X,R,K_1} \subset I_{X,R,K_0 \cup K_1}$ follow easily from the definitions. We show that for any $f \in I_{X,R,K_0 \cup K_1}$, there exists an $(f_0,f_1) \in I_{X,R,K_0} \times I_{X,R,K_1}$ such that $f = f_0 f_1$. By Proposition \ref{I is ind-projective} applied to $K_0$, there exists a $U \in \CO(X)$ such that $1_{X,R,U} f = f$ and $K_0 \cap U = \emptyset$. Put $f_0 \coloneqq 1_{X,R,U}$ and $f_1 \coloneqq f$. Then, $(f_0,f_1)$ satisfies the desired property. This implies $I_{X,R,K_0 \cup K_1} \subset I_{X,R,K_0} I_{X,R,K_1}$, which proves (i).

\vspace{0.1in}
In order to prove (ii), we show that for any $f \in I_{X,R,K_0 \cap K_1}$, there exists an $(f_0,f_1) \in I_{X,R,K_0} \times I_{X,R,K_1}$ such that $f = f_0 + f_1$ and $\n{f_0}_{\rC(X,R)} + \n{f_1}_{\rC(X,R)} \leq 2 \n{f}_{\rC(X,R)}$ under the assumption that $K_0$ and $K_1$ are closed, $X \in \TDCH$, and $R$ is isolated at $0$. Notice that the inclusion $I_{X,R,K_0} + I_{X,R,K_1} \subset I_{X,R,K_0 \cap K_1}$ follows immediately from the definition of the ideals. By Proposition \ref{Cfin is closed} (ii), we have $f^{-1}(\ens{0}) \in \CO(X)$. Put $K_2 \coloneqq K_1 \setminus f^{-1}(\ens{0})$. By $K_0 \cap K_1 \subset f^{-1}(\ens{0})$, we have $K_0 \cap K_2 = \emptyset$. By \cite{BM21} Lemma 2.6 applied to $K_0$ and $K_2$, there exists a $V \in \CO(X)$ such that $K_0 \subset V$ and $K_2 \subset X \setminus V$. By Proposition \ref{idempotent generates C(X,R)}, there exists a pair $(\cU,e)$ of a finite subset $\cU \subset \CO(X)$ and a map $e \colon \cU \to R$ such that $f = \sum_{U \in \cU} e(U) 1_{X,R,U}$. Put $f_0 \coloneqq \sum_{U \in \cU} e(U) 1_{X,R,U \setminus V}$ and $f_1 \coloneqq \sum_{U \in \cU} e(U) 1_{X,R,U \cap V}$. Then, $(f_0,f_1)$ satisfies the desired property because $\n{f}_{\rC(X,R)} = \max \ens{\n{f_0}_{\rC(X,R)},\n{f_1}_{\rC(X,R)}}$. This implies $I_{X,R,K_0 \cap K_1} \subset I_{X,R,K_0} + I_{X,R,K_1}$, and that the addition $I_{X,R,K_0} \oplus I_{X,R,K_1} \to I_{X,R,K_0 \cap K_1}$ is a strict epimorphism in $\Mod_{\cC}(\rC(X,R))$.
\end{proof}

We show an isolated analogue of \cite{BM21} Lemma 3.3.

\begin{lmm}
\label{intersection}
Let $K_0$ and $K_1$ be closed subsets of $X$. If $X \in \TDCH$ and $R$ is isolated at $0$, then the natural morphism $\rC(K_0,R) \wh{\otimes}_{\rC(X,R)}^{\bL} \rC(K_1,R) \to \rC(K_0,R) \wh{\otimes}_{\rC(X,R)} \rC(K_1,R)$ in $\Derba_{\cC}(\rC(X,R))$ is an isomorphism, and the natural morphism $\rC(K_0,R) \wh{\otimes}_{\rC(X,R)} \rC(K_1,R) \to \rC(K_0 \cap K_1,R)$ in $\Alg_{\cC}(\rC(X,R))$ is an isomorphism.
\end{lmm}

\begin{proof}
The first assertion follows from Lemma \ref{derived tensor for embedding is tensor}, and the second assertion follows from Proposition \ref{isolated at 0 implies discrete}, Proposition \ref{Cfin is closed} (i), Lemma \ref{Tietze} (ii), and Lemma \ref{subset operation and aritmetic operation} (ii).
\end{proof}

Finally, we show Theorem \ref{Tate's acyclicity for closed immersion} in a way completely parallel to the proof of \cite{BM21} Theorem 3.1.

\begin{proof}[Proof of Theorem \ref{Tate's acyclicity for closed immersion}]
The equivalence between the conditions (i) and (ii) immediately follows from the definition of a $\zeta$-cover. The equivalence between the conditions (iii) and (iv) follows from Lemma \ref{intersection} and \cite{BM21} Proposition 1.10. It remains to show that (ii) is equivalent to (iii). Let $S_0$ be a finite subset of $S$. We denote by $T_{S_0} = (T_{S_0,n},d_n)_{n \in \Z}$ the (non-derived) Tate--\Cech complex
\begin{eqnarray*}
0 \to \rC(X,R) \to \prod_{(K,j) \in S_0} \rC(K,R) \to \prod_{(K_i,j_i)_{i=0}^{1} \in [S_0]^2} \rC(K_0,R) \wh{\otimes}_{\rC(X,R)} \rC(K_1,R) \to \cdots
\end{eqnarray*}
associated to $\Gamma^{\cD_3}_*(S_0)$. Put $X_0 \coloneqq \bigcup_{(K,j) \in S_0} j(K)$. By Theorem \ref{embedding corresponds to homotopy epimorphism}, it suffices to show that $T_{S_0}$ is strictly exact in $\Mod_{\cC}(R)$ if and only if $X_0 = X$ or $R = \ens{0}$.

\vspace{0.1in}
First, suppose $X_0 = X$ or $R = \ens{0}$. If $R = \ens{0}$, then $T_{S_0}$ is a zero object, and hence is strictly exact. Therefore, we may assume $X_0 = X$ and $R \ne \{ 0 \}$. We show the strict exactness by induction on the cardinality $c$ of $S_0$. If $c \leq 2$, then the strict exactness immediately follows from the gluing lemma asserting that the openness of a subset of a topological space can be tested by the pullback by a finite closed cover. Suppose $c > 2$. Fix a $(K,j) \in S_0$. Put $S_1 \coloneqq S_0 \setminus \ens{(K,j)} \subset \Tpl/X$, and $\check{K} \coloneqq \bigcup_{(K',j') \in S_1} j'(K') \subset X$. We denote by $\check{j}$ the inclusion $\check{K} \hookrightarrow X$. Replacing $(K,j)$ by a closed immersion with the same image, we may assume $(\check{K},\check{j}) \neq (K,j)$. Put $S_2 \coloneqq \ens{(K,j)} \subset \Tpl/X$, $S_3 \coloneqq \ens{(\check{K},\check{j})} \subset \Tpl/X$, $S_4 \coloneqq \ens{(K,j),(\check{K},\check{j})} \subset \Tpl/X$, $C_0 \coloneqq T_{S_1} \wh{\otimes}_{\rC(X,R)} T_{S_2} \in \Chba(\Chba_{\cC}(\rC(X,R)))$, and $C_1 \coloneqq T_{S_3} \wh{\otimes}_{\rC(X,R)} T_{S_2} \in \Chba(\Chba_{\cC}(\rC(X,R)))$. We denote by $j$ the natural morphism $C_1 \to C_0$ in $\Chba(\Chba_{\cC}(\rC(X,R)))$. Then, $j$ is a termwise quasi-isomorphism by the induction hypothesis and $(K,\check{K}) \in \TDCH^2$ because of the finiteness of $S_1$. Therefore, $j$ induces a quasi-isomorphism $\Tot(C_1) \to \Tot(C_0)$ in $\Chba_{\cC}(\rC(X,R))$ by \cite{BM21} Proposition 1.6 (iv). Since $\Tot(C_1)$ is naturally identified with $T_{S_4}$, which is strictly exact, again by the induction hypothesis, and $\Tot(C_0)$ is naturally identified with $T_{S_0}$, $T_{S_0}$ is strictly exact.

\vspace{0.1in}
Next, suppose that $X_0 \neq X$ and $R \neq \ens{0}$. Take an $x_1 \in X \setminus X_0$. By the finiteness of $S_0$, $X_0$ is closed. By the regularity of $X$, the characteristic function $\chi \colon X_0 \cup \ens{x_1} \to R$ of $\ens{x_1}$ is continuous. By Lemma \ref{Tietze} (i) applied to $K = X_0 \cup \ens{x_1}$, there exists a $\tl{\chi} \in \rC(X,R)$ such that $\tl{\chi} |_{X_0 \cup \ens{x_1}} = \chi$. We have $\tl{\chi}(x_1) = \chi(x_1) = 1$, and hence $\tl{\chi} \neq 0$ by $R \neq \ens{0}$, while we have $\tl{\chi} |_{j(K)} = \chi |_{j(K)} = 0$ for any $(K,j) \in S_0$. Therefore, the map $\rC(X,R) \to \prod_{(K,j) \in S_0} \rC(K,R)$ in $T_{S_0}$ is not injective. This implies that $T_{S_0}$ is not strictly exact.
\end{proof}

We show an isolated analogue of \cite{BM21} Theorem 3.6 (i) in a completely parallel way.

\begin{thm}
\label{Tate's acyclicity}
Let $S \subset \Tpl/X$. If $R$ is isolated at $0$, then the following are equivalent:
\begin{itemize}
\item[(i)] The set $\Gamma^{\cD_3}_*(S)$ is a derived cover of $\rCfin(X,R)$ in $\cC$.
\item[(ii)] The set $\Gamma^{\cD_3}_*(S)$ is a non-derived cover of $\rCfin(X,R)$ in $\cC$.
\item[(iii)] The set $\set{(K,\iota_{\zeta,X} \circ j)}{(K,j) \in S}$ is a $\zeta$-cover of $\zeta(X)$ or $R = \ens{0}$.
\item[(iv)] The set $\set{(\zeta(K),\zeta(j))}{(K,j) \in S}$ is a $\zeta$-cover of $\zeta(X)$ or $R = \ens{0}$.
\item[(v)] The set $S$ is a $\zeta$-cover of $X$ or $R = \ens{0}$.
\end{itemize}
\end{thm}

\begin{proof}
The assertion immediately follows from Proposition \ref{extension property}, Theorem \ref{embedding corresponds to homotopy epimorphism}, and Theorem \ref{Tate's acyclicity for closed immersion}.
\end{proof}

As an application of Theorem \ref{Tate's acyclicity}, we show an isolated analogue of \cite{BM21} Corollary 3.6, which is itself an analogue of Kiehl's Theorem B (cf.\ \cite{Kie66} Hilfssatz 1.5) in rigid geometry.

\begin{crl}
\label{Theorem B}
Let $S \subset \Tpl/X$ be a finite set of $\zeta$-embeddings. If $R$ is isolated at $0$, then the following are equivalent:
\begin{itemize}
\item[(i)] Either $\bigcup_{(K,j) \in S} \im(\zeta(j)) = \zeta(X)$ or $R = \ens{0}$ holds.
\item[(ii)] For any $M \in \Mod_{\cC}(\rCfin(X,R))$, the (non-derived) Tate--\Cech complex
\begin{eqnarray*}
0 \to M & \to & \prod_{(K,j) \in S} \rCfin(K,R) \wh{\otimes}_{\rCfin(X,R)} M \\
& \to & \prod_{(K_i,j_i)_{i=0}^{1} \in [S]^2} \rCfin(K_0,R) \wh{\otimes}_{\rCfin(X,R)} \rCfin(K_1,R) \wh{\otimes}_{\rCfin(X,R)} M \to \cdots
\end{eqnarray*}
of $M$ associated to $S$ is strictly exact in $\Mod_{\cC}(R)$.
\item[(iii)] For any $M \in \Derba_{\cC}(\rCfin(X,R))$, the total complex of the derived Tate--\Cech bicomplex
\begin{eqnarray*}
0 \to M & \to & \prod_{(K,j) \in S} \rCfin(K,R) \wh{\otimes}_{\rCfin(X,R)}^{\bL} M \\
& \to & \prod_{(K_i,j_i)_{i=0}^{1} \in [S]^2} \rCfin(K_0,R) \wh{\otimes}_{\rCfin(X,R)}^{\bL} \rCfin(K_1,R) \wh{\otimes}_{\rCfin(X,R)}^{\bL} M \to \cdots
\end{eqnarray*}
of $M$ associated to $S$ is strictly exact in $\Mod_{\cC}(R)$.
\end{itemize}
\end{crl}

\begin{proof}
We abbreviate $\rCfin(X,R)$ to $A$. By Theorem \ref{Tate's acyclicity} and \cite{BM21} Lemma 2.20 (i), both of the conditions (ii) and (iii) applied to the case where $M$ is the regular $A$-module object imply the condition (i). We show that the condition (i) implies the conditions (ii) and (iii). If $R = \ens{0}$, then $\Mod_{\cC}(R)$ consists of zero objects. Therefore, we may assume $R \neq \ens{0}$. By Proposition \ref{extension property}, we may assume that $X$ is totally disconnected compact Hausdorff and $S$ consists closed immersions. We denote by $C = (C_n)_{n \in \Z}$ the \Cech complex
\begin{eqnarray*}
0 \to A \to \prod_{(K,j) \in S} \rCfin(K,R) \to \prod_{(K_i,j_i)_{i=0}^{1} \in [S]^2} \rCfin(K_0 \times_X K_1,R) \to \cdots
\end{eqnarray*}
associated to $S$, which is naturally isomorphic in $\Chba_{\cC}(A)$ to (and hence will be identified with) the (non-derived) Tate--\Cech complex of $A$ associated to $S$ by Lemma \ref{intersection}, and hence is strictly exact by the condition (i) and Theorem \ref{Tate's acyclicity}. For any strongly flat object $M$ of $\Mod_{\cC}(A)$, the strict exactness of $C \wh{\otimes}_{A} M$ follows from the strong flatness of $(\cdot) \wh{\otimes}_{A} M$. This implies the condition (ii) restricted to the case where $M$ is strongly flat.

\vspace{0.1in}
First, let $M \in \Derba_{\cC}(A)$. We abbreviate $(C_n \wh{\otimes}_A^{\bL} M)_{n \in \Z} \in \Chba(\Derba_{\cC}(A))$, i.e.\ the derived Tate--\Cech bicomplex of $M$ associated to $S$, to $C(M)$. The natural morphism $C(M) \to C \wh{\otimes}_{A} \cP_{\cC}^{A}(M)$ in $\Chba(\Derba_{\cC}(A))$ is represented by a morphism in $\Chba(\Chba_{\cC}(A))$ which is a termwise quasi-isomorphism by the termwise strong flatness of $\cP_{\cC}^{A}(M)$ (cf.\ \cite{BK20} Proposition 3.11), and the zero morphism $C \wh{\otimes}_{A} \cP_{\cC}^{A}(M) \to 0$ in $\Chba(\Chba_{\cC}(A))$ is a termwise quasi-isomorphism by the condition (ii) for strongly flat objects. This implies that the zero morphism $\Tot(C(M)) \to 0$ in $\Chba_{\cC}(A)$ is a quasi-isomorphism by Proposition 1.6 (iv). Therefore, $\Tot(C(M))$ is strictly exact by \cite{BM21} Proposition 1.6 (iii). This implies the condition (iii).

\vspace{0.1in}
Next, let $M \in \Mod_{\cC}(A)$. We denote by $C(M) \in \Chba(\Chba_{\cC}(A))$ the canonical representative 
\begin{eqnarray*}
0 & \to & \cP_{\cC}^{A}(M) \to \prod_{(K,j) \in S} \cP_{\cC}^{A}(\rCfin(K,R)) \wh{\otimes}_{A} \cP_{\cC}^{A}(M) \\
& \to & \prod_{(K_i,j_i)_{i=0}^{1} \in [S]^2} \cP_{\cC}^{A}(\cP_{\cC}^{A}(\rCfin(K_0,R)) \wh{\otimes}_{A} \cP_{\cC}^{A}(\rCfin(K_1,R))) \wh{\otimes}_{A} \cP_{\cC}^{A}(M) \to \cdots
\end{eqnarray*}
of the derived Tate--\Cech bicomplex of $M$ regarded as an object of $\Chba(A)$ concentrated in degree $0$. Then, the natural morphism $p \colon C(M) \to C \wh{\otimes}_{A} M$ in $\Chba(\Chba_{\cC}(A))$ represents a morphism in $\Chba(\Derba_{\cC}(A))$ which is a termwise isomorphism by Lemma \ref{derived tensor for embedding is tensor} and Lemma \ref{intersection}, and hence is a termwise quasi-isomorphism by \cite{BM21} Corollary 1.5 (i). This implies that $\Tot(p) \colon \Tot(C(M)) \to C \wh{\otimes}_{A} M$ is a quasi-isomorphism by \cite{BM21} Proposition 1.6 (iv). Since we have shown that the condition (i) implies the condition (iii), $\Tot(C(M))$ is strictly exact. Therefore, the (non-derived) Tate--\Cech complex $C \wh{\otimes}_{A} M$ of $M$ is strictly exact.
\end{proof}

\section{Effective descent for modules over Banach algebras of functions valued in discrete Banach rings}
\label{Effective descent for modules over Banach algebras of functions valued in discrete Banach rings}

In this section, we prove effective descent for the categories of Banach modules over the algebras $\rCfin(X,R)$, and their derived categories, with respect to the topology determined by finite families of $\zeta$-embeddings. As we have proved the analogues of the main results of \cite{BM21} for the algebras $\rCfin(X,R)$, for $R$ isolated at $0$, the proof of effective descent given here will follow closely that presented in \cite{BM21} \S 4.

\vspace{0.1in}
Suppose that $R$ is isolated at $0$. Although $\cD_3$ in \cite{BM21} \S 4 was assumed to be an admissible triple, the formulation of derived and non-derived effective descent works for $\rCfin(X,R)$ in a completely parallel way. Since we omit overlapping explanations here, except for definitions and notations, it is better for the reader to refer to \cite{BM21} \S 4 for details.

\vspace{0.1in}
For a $(Y,\Phi) \in \Tpl/X$, we consider the monads $\cTnd(\Phi) = \Phi_* \Phi^*$ on $\Mod_{\cC}(\rCfin(X,R))$ and $\cTd(\Phi) = \Phi_* \bL \Phi^*$ on $\Derba_{\cC}(\rCfin(X,R))$ defined by the adjunctions
\begin{eqnarray*}
\Phi^* \colon \Mod_{\cC}(\rCfin(X,R)) & \leftrightarrows & \Mod_{\cC}(\rCfin(Y,R)) \colon \Phi_* \\
\bL \Phi^* \colon \Derba_{\cC}(\rCfin(X,R)) & \leftrightarrows & \Derba_{\cC}(\rCfin(Y,R)) \colon \Phi_*,
\end{eqnarray*}
where $\Phi^*$ is the extension of scalars functor $(\cdot) \wh{\otimes}_{\rCfin(X,R)} \rCfin(Y, R)$, $\bL \Phi^*$ is the derived extension of scalars functor $(\cdot) \wh{\otimes}_{\rCfin(X,R)}^{\bL} \rCfin(Y, R)$, and $\Phi_*$ is the restriction of scalars functor.

\vspace{0.1in}
Let $\cT$ denote $\cTnd(\Phi)$ (resp.\ $\cTd(\Phi)$), and abbreviate $\Mod_{\cC}$ (resp.\ $\Derba_{\cC}$) to $\rM$. There is an adjunction
\begin{eqnarray*}
F_{\cT} \colon \rM(\rCfin(X,R)) & \leftrightarrows & \Mod(\cT) \colon U_{\cT},
\end{eqnarray*}
where $\Mod(\cT)$ is the category of modules over the monad $\cT$, $F_{\cT}$ is the free module functor, and $U_{\cT}$ is the forgetful functor. We denote by $L_{\cT}$ the comonad $F_{\cT} \circ U_{\cT}$, and by $\Desc_{\rM(\rCfin(X,R))}(\Phi)$ the category of left $L_{\cT}$-comodule objects of $\Mod(\cT)$. An object of $\Desc_{\rM(\rCfin(X,R))}(\Phi)$ can be described as a triple $(M, \rho, \sigma)$ of an $M \in \rM(\rCfin(X,R))$ and morphisms $\rho \colon \cT(M) \to M$ and $\sigma \colon M \to \cT(M)$ in $\rM(\rCfin(X,R))$ that give respectively the left action of $\cT$ and the left coaction of $L_{\cT}$ and that satisfy some further compatibility properties. See \cite{Bal12} Remark 1.4 for details.

\vspace{0.1in}
We define the {\it comparison functor} $Q_{\cT}$ as the functor
\begin{eqnarray*}
\rM(\rCfin(X,R)) \to \Desc_{\rM(\rCfin(X,R))}(\Phi)
\end{eqnarray*}
assigning to each $M \in \rM(\rCfin(X,R))$ the tuple $(\cT(M), \epsilon_M, \cT(\eta_M))$, where $\epsilon$ is the counit $\cT \Rightarrow \id_{\rM(\rCfin(X,R))}$ of the adjunction and $\eta$ is the unit $\id_{\rM(\rCfin(X,R))} \Rightarrow \cT$ of the adjunction. We say that $\Phi$ {\it satisfies non-derived} (resp.\ {\it derived}) {\it effective descent} if the comparison functor $Q_{\cTd}$ (resp.\ $Q_{\cTnd}$) is an equivalence. We will apply the notion of the descent category in the following specific case.

\vspace{0.1in}
Let $S \subset \Tpl/X$. We put $Y \coloneqq \coprod_{(K,j) \in S} K$, and denote by $\Phi$ the canonical morphism $Y \to X$ induced by the universal property of the coproduct of $(j)_{(K,j) \in S}$. We say that $S$ satisfies {\it non-derived} (resp.\ {\it derived}) {\it effective descent} if so does $\Phi$. As an isolated analogue of \cite{BM21} Theorem 4.1, which itself is an analogue of the equivalence between finite Banach modules over an affinoid algebra and coherent sheaves over the rigid analytic space associated to it (cf.\ \cite{Kie66} Theorem 1.2), we give our main result on derived and non-derived effective descent in a completely parallel way.

\begin{thm}
\label{derived effective descent}
If $S$ consists of $\zeta$-embeddings and $R$ is isolated at $0$, then the following are equivalent:
\begin{itemize}
\item[(i)] The set $S$ is a $\zeta$-cover of $X$ or $R = \ens{0}$.
\item[(ii)] The functor $\bL \Phi^*$ is faithful.
\item[(iii)] The functor $\bL \Phi^*$ is conservative.
\item[(iv)] The functor $\Phi^*$ is faithful.
\item[(v)] The functor $\Phi^*$ is conservative.
\item[(vi)] The set $S$ satisfies derived effective descent.
\item[(vii)] The set $S$ satisfies non-derived effective descent.
\end{itemize}
\end{thm}

\begin{proof}
The equivalence between the condition (ii) and the condition (vi), the implication from the condition (ii) to the condition (iii), and the implication from the condition (vii) to the condition (iv) follow from arguments completely parallel to those in the first paragraph of the proof of \cite{BM21} Theorem 4.1.

\vspace{0.1in}
The implication from the condition (ii) (resp.\ (iii)) to the condition (iv) (resp.\ (v)) follows from an argument completely parallel to that in the second paragraph of the proof of \cite{BM21} Theorem 4.1 except that we replace the fact that the natural transform $j_S$ is a natural isomorphism based on \cite{BM21} Lemma 2.26 by the direct application of Lemma \ref{derived tensor for embedding is tensor}.

\vspace{0.1in}
The implication from the condition (iv) (resp.\ (v)) to the condition (i) follows from an argument completely parallel to that in the third paragraph of the proof of \cite{BM21} Theorem 4.1 except that we replace \cite{BM21} Lemma 3.3 by its isolated analogue Lemma \ref{intersection}, because a morphism $\rC(\ens{x},R) \to \rC(\ens{x},R)$ in $\Mod_{\cC}(R)$ for an $x \in X$ cannot be simultaneously the zero morphism and the identity morphism unless $R = \ens{0}$.

\vspace{0.1in}
The implications from the condition (i) to the conditions (ii) and (vii) follow from arguments completely parallel to those in the fourth and rest paragraphs of the proof of \cite{BM21} Theorem 4.1 except that we replace \cite{BM21} Corollary 3.4 by its isolated analogue Corollary \ref{Theorem B}.
\end{proof}

\section{Homological algebra for discrete Banach modules}
\label{Homological algebra for discrete Banach modules}

Up to this point, we have proved analogues of results in \cite{Mih14} and \cite{BM21}. From this section on, we provide new results and new properties of the algebras of continuous functions valued on a Banach ring $R$ isolated at $0$. We define the category of Banach $R$-modules isolated at $0$. We prove in Proposition \ref{Modiso is quasi-Abelian} that this category is quasi-Abelian, and in Proposition \ref{Modiso is monoidal subcategory} that it is a closed symmetric monoidal subcategory of the category of all Banach $R$-modules. We also describe its bounded above derived category as a full monoidal subcategory of the category $\Derba_\cC(R)$ and determine its generators in Propositions \ref{isolated free module}. Then, in Theorem \ref{absorbing law}, we prove that the canonical morphism $\rCfin(X, M_1) \wh{\otimes}_R M_2 \to \rCfin(X, M_1 \wh{\otimes}_R M_2)$ is always an isomorphism when the non-Archimedean projective tensor product is considered (and hence all modules are non-Archimedean), and we give counterexamples for the Archimedean tensor product.

\vspace{0.1in}
We denote by $\Modiso_{\cC}(R) \subset \Mod_{\cC}(R)$ and $\Algiso_{\cC}(R) \subset \Alg_{\cC}(R)$ the full subcategories of objects isolated at $0$. We notice that Proposition \ref{isolated at 0 implies discrete} implies that these inclusions of categories are also replete, i.\ e.\ they are closed by isomorphisms of objects. By Lemma \ref{Cfin is closed} (i), the correspondence $(X,M) \mapsto \rCfin(X,M)$ gives functors $\rCfin^{\iso} \colon \Tpl \times \Modiso_{\cC}(R) \to \Modiso_{\cC}(R)$ and $\rCfin^{\iso} \colon \Tpl \times \Algiso_{\cC}(R) \to \Algiso_{\cC}(R)$. We denote by $F^{\iso}$ the inclusion functor $\Modiso_{\cC}(R) \hookrightarrow \Mod_{\cC}(R)$. First, we study several operations on $\Mod_{\cC}(R)$ restricted to $\Modiso_{\cC}(R)$.

\begin{prp}
\label{Modiso is quasi-Abelian}
\begin{itemize}
\item[(i)] Let $M \in \Modiso_{\cC}(R)$. For any $R$-submodule $N \subset M$, $N$ (resp.\ $M/N$) forms an object of $\Modiso_{\cC}(R)$ with respect to the restriction of $\n{\cdot}_M$ (resp.\ the quotient seminorm).
\item[(ii)] The functor $F^{\iso}$ preserves and reflects kernels and cokernels.
\item[(iii)] The category $\Modiso_{\cC}(R)$ is closed by finite direct sums in $\Mod_\cC(R)$.
\item[(iv)] The category $\Modiso_{\cC}(R)$ is quasi-Abelian.
\end{itemize}
\end{prp}

We recall that a functor $G \colon \cA \to \cB$ between additive categories is said to \emph{reflect kernels} (resp.\ \emph{reflect cokernels}) if given a morphism $f$ in $\cA$ such that $G(f)$ is a strict monomorphism (resp.\ strict epimorphism) then $f$ is a strict monomorphism (resp. strict epimorphism).

\begin{proof}
The assertion (i) follows from Proposition \ref{isolated at 0 implies discrete}, the assertion (ii) follows from the assertion (i). Let us show the assertion (iii). Let $M, N \in \Modiso_{\cC}(R)$, then $M \oplus N \cong (M \times N, \v{\cdot}_{\max{}})$, where $|(m,n)|_{\max{}} = \max \{ |m|, |n|\}$. Therefore, $M \oplus N \in \Modiso_{\cC}(R)$. Then, the assertion (iv) follows from the assertion{s} (ii) and (iii) because $\Mod_{\cC}(R)$ is quasi-Abelian.
\end{proof}

If $R$ is isolated at $0$, then the category $\Modiso_{\cC}(R)$ is obviously non-trivial. We provide some basic examples of objects of $\Modiso_{\cC}(R)$ in order to show that this category can be non-trivial also when $R$ is not isolated at $0$.

\begin{exm}
\label{modules isolated at zero}
\begin{itemize}
\item[(i)] The fundamental examples of Banach rings isolated at zero are $\Z_\infty$ and $\Z_{\triv}$, having rich categories of modules isolated at zero.
\item[(ii)] If $R$ is a complete (Archimedean or non-Archimedean) non-trivially valued field, then the category $\Modiso_{\cC}(R)$ consists of only zero objects.
\item[(iii)] If $R$ is $\Zp$ equipped with the $p$-adic absolute value, then $\F_p$ equipped with the trivial norm is an object of $\Modiso_{\cC}(R)$, via the identification $\F_p \cong \Z_p/(p)$. Therefore, $\Modiso_{\cC}(R)$ contains the category of $\F_p$-vector spaces equipped with the trivial norm.
\item[(iv)] If $k$ is an arbitrary field and $k \dbrack{T}$ is the ring of formal power-series over $k$ equipped with the $T$-adic norm, then $k$ equipped with the trivial norm is a Banach module over $k \dbrack{T}$ isolated at $0$ via the identification $k \cong k \dbrack{T}/(T)$.
\end{itemize}
\end{exm}

Suppose that $R$ is isolated at $0$. Then, for any $M \in \Modiso_{\cC}(R)$, $\rCfin(X,M) \in \Modiso_{\cC}(R)$ forms an object of $\Modiso_{\cC}(\rCfin(X,R))$ with respect to the pointwise scalar multiplication. The correspondence $M \mapsto \rCfin(X,M)$ also gives functors $\rCfin^{\iso}(X,\cdot) \colon \Modiso(R) \to \Modiso(\rCfin(X,R))$ and $\rCfin^{\iso}(X,\cdot) \colon \Algiso(R) \to \Algiso(\rCfin(X,R))$. We recall that for any $M \in \Mod_{\cC}(R)$, $\rCfin(X,M)$ forms a $\rCfin(X,R)$-module even when $R$ or $M$ is not isolated at $0$, and hence we use the algebraic module structure without such an assumption, as we did in \S \ref{Berkovich spectrum for functions valued in discrete Banach rings}.

\vspace{0.1in}
By Proposition \ref{Modiso is quasi-Abelian} (ii), $F^{\iso}$ preserves and reflects the strict exactness of a null sequence, and hence the class of strict exact sequences of $\Modiso_{\cC}(R)$ agrees with the restriction to $\Modiso_{\cC}(R)$ of the class of strict exact sequences of $\Mod_{\cC}(R)$. Nevertheless, it is meaningful to consider $\Modiso_{\cC}(R)$, because, for example, the following proposition does not make sense for $\Mod_{\cC}(R)$.

\begin{prp}
\label{Cfin is strongly exact}
If $R$ is isolated at $0$, then $\rCfin^{\iso}(X,\cdot)\colon \Modiso(R) \to \Modiso(\rCfin(X,R))$ is a strongly exact functor.
\end{prp}

\begin{proof}
Let $E$ be a null sequence $M_0 \stackrel{s}{\to} M_1 \stackrel{t}{\to} M_2$ of $\Modiso_{\cC}(R)$. It suffices to show that the following two claims:
\begin{itemize}
\item[(i)] If $s$ is a strict monomorphism and $E$ is exact, then $\rCfin^{\iso}(X,s)$ is a strict monomorphism and $\rCfin^{\iso}(X,E)$ is exact.
\item[(ii)] If $t$ is a strict epimorphism and $E$ is exact, then $\rCfin^{\iso}(X,t)$ is a strict epimorphism and $\rCfin^{\iso}(X,E)$ is exact.
\end{itemize}
The claim (i) follows from the fact that $\ker(\rCfin^{\iso}(X,t))$ is given as
\begin{eqnarray*}
\set{f \in \rCfin(X,M_1)}{t \circ f = 0} = \set{f \in \rCfin(X,M_1)}{\forall x \in X, f(x) \in \ker(t)}
\end{eqnarray*}
equipped with the restriction of $\n{\cdot}_{\rCfin(X,M_1)}$. We show the claim (ii). Since $t$ is a strict epimorphism, there exists a $C_1 \in (0,\infty)$ such that for any $m \in M_2$, there exists an $\tl{m} \in M_1$ such that $t(\tl{m}) = m$ and $\n{\tl{m}}_{M_1} \leq C_1 \n{m}_{M_2}$.

\vspace{0.1in}
Let $f \in \rCfin(X,M_2)$. By Proposition \ref{idempotent generates C(X,R)} (ii), there exists a pair $(\cU,e)$ of a finite subset $\cU \subset \CO(X) \setminus \ens{\emptyset}$ and a map $e \colon \cU \to M_2$ such that $\cU$ is a disjoint cover of $X$ and $\sum_{U \in \cU} 1_{X,R,U} e(U) = f$. Then, we have $\n{f}_{\rCfin(X,M_2)} = \sup_{U \in \cU} \n{e(U)}_{M_2}$. By the choice of $C_1$, there exists a map $\tl{e} \colon \cU \to M_1$ such that for any $U \in \cU$, $t(\tl{e}(U)) = e(U)$ and $\n{\tl{e}(U)}_{M_1} \leq C_1 \n{e(U)}_{M_2}$. Put $\tl{f} \coloneqq \sum_{U \in \cU} 1_{X,R,U} \tl{e}(U) \in \rCfin(X,M_1)$. Then, we have $\rCfin^{\iso}(X,t)(\tl{f}) = t \circ \tl{f} = f$ and $\n{\tl{f}}_{\rCfin(X,M_1)} = \sup_{U \in \cU} \n{\tl{e}(U)}_{M_1} \leq \sup_{U \in \cU} C_1 \n{e(U)}_{M_2} = C_1 \n{f}_{\rCfin(X,M_2)}$. This implies that $\rCfin^{\iso}(X,t)$ is a strict epimorphism.

\vspace{0.1in}
Let $f \in \ker(\rCfin^{\iso}(X,t))$. By $\im(s) = \ker(t)$, there exists a $C_0 \in (0,\infty)$ such that for any $(m,\epsilon) \in \ker(t) \times (0,\infty)$, there exists an $\tl{m} \in M_0$ such that $\n{s(\tl{m}) - m}_{M_1} < \epsilon$ and $\n{\tl{m}}_{M_0} \leq C_0 \n{m}_{M_1}$. By Proposition \ref{idempotent generates C(X,R)} (ii), there exists a pair $(\cU,e)$ of a finite subset $\cU \subset \CO(X) \setminus \ens{\emptyset}$ and a map $e \colon \cU \to M_1$ such that $\cU$ is a disjoint cover of $X$ and $\sum_{U \in \cU} 1_{X,R,U} e(U) = f$. By the choice of $C_0$, there exists a map $\tl{e} \colon \cU \to M_0$ such that for any $U \in \cU$, $\n{s(\tl{e}(U)) - e(U)}_{M_0} < \epsilon$ and $\n{\tl{e}(U)}_{M_0} \leq C_0 \n{e(U)}_{M_1}$. Put $\tl{f} \coloneqq \sum_{U \in \cU} 1_{X,R,U} \tl{e}(U) \in \rCfin(X,M_0)$. Then, we have $\rCfin^{\iso}(X,s)(\tl{f}) = s \circ \tl{f} = f$ and $\n{\tl{f}}_{\rCfin(X,M_0)} = \sup_{U \in \cU} \n{\tl{e}(U)}_{M_0} \leq \sup_{U \in \cU} C_0 \n{e(U)}_{M_2} = C_0 \n{f}_{\rCfin(X,M_1)}$. This implies that $\rCfin^{\iso}(X,E)$ is exact.
\end{proof}

Let $R_0$ be a ring. We denote by $\Mod(R_0)$ the category of $R_0$-modules and $R_0$-linear homomorphisms. When $R_0$ is the underlying ring of $R$, then we put $\Mod_{\Ab}(R) \coloneqq \Mod(R_0)$. We denote by $F^{\Ab}$ the forgetful functor $\Modiso_{\cC}(R) \to \Mod_{\Ab}(R)$. 

\begin{prp}
\label{Modiso is monoidal subcategory}
\begin{itemize}
\item[(i)] For any $(M_0,M_1) \in \Modiso_{\cC}(R)^2$, the natural $R$-linear homomorphism $\varphi \colon M_0 \otimes_R M_1 \to M_0 \wh{\otimes}_R M_1$ is bijective, and $M_0 \wh{\otimes}_R M_1$ is isolated at $0$.
\item[(ii)] The category $\Modiso_{\cC}(R)$ is a symmetric monoidal subcategory of $\Mod_{\cC}(R)$.
\item[(iii)] The functor $F^{\Ab}$ is a symmetric monoidal functor.
\end{itemize}
\end{prp}

\begin{proof}
The assertions (ii) and (iii) immediately follow from the assertion (i). We show the assertion (i).
Since $M_0$ and $M_1$ are isolated at $0$, there exists an $\epsilon \in (0,\infty)$ such that for any $i \in \ens{0,1}$, the equality $\set{m \in M_i}{\n{m}_M < \epsilon} = \ens{0}$ holds. Then, for any $m \in M_0 \otimes_R M_1$, if the tensor seminorm of $m$ is smaller than $\epsilon^2$, then $m = 0$. Therefore, the $R$-linear homomorphism $\varphi$ is injective, and the equality $\set{m \in M_0 \wh{\otimes}_R M_1}{\n{m}_{M_0 \wh{\otimes}_R M_1} < \epsilon^2} = \ens{0}$ holds. By Proposition \ref{isolated at 0 implies discrete}, $\varphi$ is surjective. 
\end{proof}

We compare the quotient by a closed ideal and the base change by the quotient ring for a Banach module isolated at $0$.

\begin{prp}
\label{Modiso is closed under quotient}
Let $M \in \Modiso_{\cC}(R)$. For any closed ideal $I \subset R$, $M \wh{\otimes}_R R/I$ is isolated at $0$, and the natural morphism $R/I \wh{\otimes}_R M \to M/I M$ in $\Mod_{\cC}(R)$ is an isomorphism.
\end{prp}

\begin{proof}
By Proposition \ref{Modiso is quasi-Abelian} (i), $M/I M$ forms an object of $\Modiso_{\cC}(R)$, and hence the first claim follows from the second claim by Proposition \ref{isolated at 0 implies discrete}. We denote by $i$ the inclusion $I \hookrightarrow R$, and by $E$ the strictly coexact sequence $I \stackrel{i}{\to} R \to R/I \to 0$ of $\Mod_{\cC}(R)$. By the strong right exactness of $(\cdot) \wh{\otimes}_R M$, $E \wh{\otimes}_R M$ is strictly coexact. By Proposition \ref{isolated at 0 implies discrete}, the image of $i \wh{\otimes}_R M$ in $\Mod_{\cC}(R)$ coincides with $I M$. Therefore, the strict coexactness of $E \wh{\otimes}_R M$ implies that the morphism in the assertion is an isomorphism.
\end{proof}

For an $M \in \Mod_{\cC}(R)$, we say that $\n{\cdot}_M$ is {\it submultiplicative over $R$} if the inequality $\n{am}_M \leq \n{a}_R \n{m}_M$ holds for any $(a,m) \in A \times M$. We recall that $\n{\cdot}_R$ is equivalent to the submultiplicative norm over $R$ assigning $\sup_{b \in R \setminus \ens{0}} \frac{\n{ab}_R}{\n{b}_R}$ to each $a \in R$ by \cite{BGR84} Proposition 1.2.1/2, and if $\n{\cdot}_R$ itself is submultiplicative over $R$, then for any $M \in \Mod_{\cC}(R)$, $\n{\cdot}_M$ is equivalent to the submultiplicative norm over $R$ assigning $\sup_{a \in R \setminus \ens{0}} \frac{\n{am}_M}{\n{a}_R}$ to each $m \in M$.

\vspace{0.1in}
A family $(M_i)_{i \in I}$ of objects in a category is said to be {\it essentially finite} if $M_i$ is a zero object for all but finitely many $i \in I$. For any family of objects $(M_i)_{i \in I}$ in $\Mod_{\cC}(R)$ with a uniform bound of the operator norm of the scalar multiplication, we defined in \cite{BM21} p.\ 10 the {it contracting direct sum} $\bigoplus_{i \in I}^{\cC_{\leq 1}} M_i$ and the {it contracting direct product} $\prod_{i \in I}^{\cC_{\leq 1}} M_i$. In particular, they are defined for any family of objects in $\Mod_{\cC}(R)$ with submultiplicative norms over $R$. These constructions are not functorial for bounded homomorphisms but for contracting homomorphisms. The restrictions of these two assignments define functors from the categories of essentially finite discrete diagrams of objects of $\Mod_{\cC}(R)$ and they define naturally isomorphic direct sum functors. Therefore, we will drop the superscript $\cC_{\leq 1}$ and denote the direct sum $\bigoplus_{i \in I}^{\cC_{\leq 1}} M_i$ as $\bigoplus_{i \in I} M_i$ for essentially finite families $(M_i)_{i \in I}$ of objects of $\Mod_{\cC}(R)$. Notice that all additive functors preserve essentially finite direct sums.

\vspace{0.1in}
We recall that a normed set is a set $S$ equipped with a map $\v{\cdot}_S \colon S \to [0,\infty)$, and $\NSet$ denotes the category of normed sets and bounded maps. See \cite{BBK19} \S 1 for more details about this category. For an $S \in \NSet$, we put $S_+ \coloneqq \set{s \in S}{\n{s}_S > 0}$. For example, every subset $S$ of an $M \in \Mod_{\cC}(R)$ is a normed set with respect to the norm and satisfies $S_+ = S \setminus \ens{0}$. An $S \in \NSet$ is said to be {\it isolated at $0$} if there exists an $\epsilon \in (0,\infty)$ such that $\set{s \in S}{\n{s}_S < \epsilon} = S \setminus S_+$.

\begin{exm}
\label{trivially normed set}
\begin{itemize}
\item[(i)] Every subset of an object of $\Modiso_{\cC}(R)$ is a normed set isolated at $0$.
\item[(ii)] Every set $S$ forms a normed set isolated at $0$ with respect to the constant map $1_S \colon S \to [0,\infty)$ with value $1$.
\end{itemize}
\end{exm}

For an $S \in \NSet$, we denote by $\ell^{\cC}(S,R) \in \Mod_{\cC}(R)$ the projective object introduced in \cite{BM21} p.\ 12 (cf.\ \cite{BB16} Lemma 3.27 and \cite{BK17} Lemma A.39), i.e.\ the completion of $R^{\oplus S}$ with respect to the seminorm assigning to each $(r_s)_{s \in S} \in R^{\oplus S}$ the value $\sum_{s \in S} \n{r_s}_R \v{s}_S \in [0,\infty)$ when $\cC = \BAb$ and $\max_{s \in S} \n{r_s}_R \v{s}_S \in [0,\infty)$ when $\cC = \NBAb$. We put $\Chisoba_{\cC}(R) \coloneqq \Chba(\Modiso_{\cC}(R)) \subset \Chba_{\cC}(R)$. 

\begin{prp}
\label{isolated free module}
Let $S \in \NSet$.
\begin{itemize}
\item[(i)] If $S$ and $R$ are isolated at $0$, then the natural $R$-linear homomorphism $R^{\oplus S_+} \to \ell^{\cC}(S,R)$ is bijective, and $\ell^{\cC}(S,R)$ is a projective object of $\Modiso_{\cC}(R)$.
\item[(ii)] Let $A \in \Alg_{\cC}(R)$. The natural morphism $\varphi \colon \ell^{\cC}(S,R) \wh{\otimes}_R A \to \ell^{\cC}(S,A)$ in $\Mod_{\cC}(A)$ is an isomorphism. In addition, if $\n{1}_R = 1$ and $\n{\cdot}_A$ is submultiplicative over $R$, then it is an isometric $A$-linear isomorphism.
\item[(iii)] If $S$ and $R$ are isolated at $0$, then for any $A \in \Algiso_{\cC}(R)$, the natural $A$-linear homomorphism $\ell^{\cC}(S,R) \otimes_R A \to \ell^{\cC}(S,A)$ is bijective.
\end{itemize}
\end{prp}

\begin{proof}
The assertion (iii) follows from the assertions (i) and (ii), and by Proposition \ref{Modiso is monoidal subcategory} (i). We show the assertion (i).
For any $(r_s)_{s \in S} \in R^{\oplus S}$, the seminorm of $(r_s)_{s \in S}$ is given as $\sum_{s \in S} \n{r_s}_R \v{s}_S = \sum_{s \in S_+} \n{r_s}_R \v{s}_S$, and hence the completion $\ell^{\cC}(S,R)$ of $R^{\oplus S}$ is naturally identified with the completion of $R^{\oplus S_+}$ equipped with the norm given by the restriction of the seminorm. By $\set{r \in R^{\oplus S_+}}{\n{r}_{R^{\oplus S_+}} < \epsilon^2} = \ens{0}$, $R^{\oplus S_+}$ is complete and isolated at $0$. This implies that the morphism in the assertion (i) is an isomorphism in $\Mod_{\cC}(R)$, and hence $\ell^{\cC}(S,R)$ is isolated at $0$.
The bijectivity of the natural $R$-linear homomorphism follows from Proposition \ref{isolated at 0 implies discrete}. The projectivity of $\ell^{\cC}(S,R)$ in $\Modiso_{\cC}(R)$ follows from the projectivity in $\Mod_{\cC}(R)$, because the inclusion $\Modiso_{\cC}(R) \hookrightarrow \Mod_{\cC}(R)$ preserves the strict exactness of a null sequence by Proposition \ref{Modiso is quasi-Abelian} (ii).

\vspace{0.1in}
The morphism $\varphi$ in the assertion (ii) is a bounded isomorphism by the universality of $\wh{\otimes}_R$ for bounded $R$-bilinear homomorphisms (cf.\ \cite{BM21} p.\ 8) and the adjoint property of $\ell^{\cC}(\cdot,R)$ for bounded maps (cf.\ \cite{BM21} p.\ 12). Therefore, it suffices to show that $\varphi$ is an isometry under the assumption that $\n{1}_R = 1$ and $\n{\cdot}_A$ is submultiplicative over $R$. For this purpose, it suffices to show that for any $(a_s)_{s \in S_+} \in A^{\oplus S_+}$, the tensor seminorm $C$ of an element $\sum_{s \in S_+} \delta_s \otimes a_s \in R^{\oplus S_+} \otimes_R A$ with respect to the restriction of $\n{\cdot}_{\ell^{\cC}(S,R)}$ to $R^{\oplus S_+}$ coincides with $\sum^{\cC}_{s \in S_+} \n{a_s}_A \v{s}_S$, where $\delta_s \in R^{\oplus S_+}$ denotes the image of $s \in S_+$ by the canonical embedding $S_+ \hookrightarrow R^{\oplus S_+}$, and $\sum^{\cC}$ denotes $\sum$ when $\cC = \BAb$ and $\sup$ when $\cC = \NBAb$. We have $C \leq \sum^{\cC}_{s \in S_+} \n{\delta_s}_{\ell^{\cC}(S,R)} \n{a_s}_A = \sum^{\cC}_{s \in S_+} \n{1}_R \v{s}_S \n{a_s}_A = \sum^{\cC}_{s \in S_+} \n{a_s}_A \v{s}_S$.

\vspace{0.1in}
Let $((r_{i,s})_{s \in S_+},b_i)_{i=0}^{n} \in (R^{\oplus S_+},A)^{n+1}$ with $n \in \N$ and $\sum_{i=0}^{n} (r_{i,s})_{s \in S_+} \otimes b_i = \sum_{s \in S_+} \delta_s \otimes a_s$. For any $s \in S_+$, we have $\sum_{i=0}^{n} r_{i,s} b_i = a_s$, and hence $\sum^{\cC,n}_{i=0} \n{r_{i,s}}_R \n{b_i}_A \geq \sum^{\cC,n}_{i=0} \n{r_{i,s} b_i}_A \geq \n{a_s}_A$ by the submultiplicativity of $\n{\cdot}_A$ over $R$ and by the triangle inequality when $\cC = \BAb$ and the strong triangle inequality when $\cC = \NBAb$. We obtain
\begin{eqnarray*}
& & \sum^{\cC,n}_{i=0} \n{(r_{i,s})_{s \in S_+}}_{\ell^{\cC}(S,R)} \n{b_i}_A \geq \sum^{\cC,n}_{i=0} \sum^{\cC}_{s \in S_+} \n{r_{i,s}}_R \v{s}_S \n{b_i}_A \\
& = & \sum^{\cC}_{s \in S_+} \left( \sum^{\cC,n}_{i=0} \n{r_{i,s}}_R \n{b_i}_A \right) \v{s}_S \geq \sum^{\cC}_{s \in S_+} \n{a_s}_A \v{s}_S.
\end{eqnarray*}
This implies $C \geq \sum^{\cC}_{s \in S_+} \n{a_s}_A \v{s}_S$.
\end{proof}

By Proposition \ref{isolated free module} (i), we obtain the following.

\begin{crl}
\label{isolated projective module}
\begin{itemize}
\item[(i)] If $R$ is isolated at $0$, then $\Modiso_{\cC}(R)$ has enough projective objects.
\item[(ii)] If $R$ is isolated at $0$, then $\Chisoba_{\cC}(R)$ is stable under the functor $\cP_{\cC}^R$, and the restriction of $\cP_{\cC}^R$ gives a projective resolution functor on $\Chisoba_{\cC}(R)$.
\end{itemize}
\end{crl}

We denote by $\Derisoba_{\cC}(R)$ the derived category of $\Chisoba_{\cC}(R)$, i.e.\ the localisation of the homotopy category of $\Chisoba_{\cC}(R)$ by the multiplicative system of morphisms represented by quasi-isomorphisms. By the universality of localisation, the strongly exact inclusion $\Chisoba_{\cC}(R) \hookrightarrow \Chba_{\cC}(R)$ induces a functor $F^{\iso} \colon \Derisoba_{\cC}(R) \to \Derba_{\cC}(R)$.

\vspace{0.1in}
Suppose that $R$ is isolated at $0$. Then, for any $(M_0,M_1) \in \Chisoba_{\cC}(R)^2$, the total complex $\Tot(\cP_{\cC}^R(M_0) \wh{\otimes}_R \cP_{\cC}^R(M_1))$ of $M_0 \wh{\otimes}_R M_1$ with respect to the direct sum in $\Modiso_{\cC}(R)$ makes sense as an object of $\Chisoba_{\cC}(R)$, because the direct sums are essentially finite and we can apply Proposition \ref{Modiso is monoidal subcategory} (i), Proposition
\ref{Modiso is quasi-Abelian} (iii), and Proposition \ref{isolated free module} (i). The correspondence $(M_0,M_1) \mapsto \Tot(M_0 \wh{\otimes}_R^{\bL} M_1)$ gives a functor $\wh{\otimes}_R^{\iso \bL} \colon \Derisoba_{\cC}(R)^2 \to \Derisoba_{\cC}(R)$, which is an isolated analogue of $\wh{\otimes}_R^{\bL}$, by Corollary \ref{isolated projective module} (ii).

\begin{prp}
\label{isolated derived tensor}
If $R$ is isolated at $0$, then the tuple $(\Derisoba_{\cC}(R),\wh{\otimes}_R^{\iso \bL},R)$ naturally forms a symmetric monoidal category, and $F^{\iso} \colon \Derisoba_{\cC}(R) \to \Derba_{\cC}(R)$ is a symmetric monoidal fully faithful functor.
\end{prp}

\begin{proof}
The first assertion follows from the symmetry and the associativity of the total complex, and the fact that $F^{\iso}$ is a symmetric monoidal functor immediately follows from the construction of $\wh{\otimes}_R^{\iso \bL}$ using the same representative as a chain complex. It remains to show that $F^{\iso}$ is fully faithful. Using the dual of \cite{Buh10} Corollary 10.22 (ii), it is enough to check that whenever we have a strictly exact sequence
\begin{eqnarray*}
0 \to M \to N \to L \to 0
\end{eqnarray*}
of $\Mod_{\cC}(R)$ with $L \in \Modiso_{\cC}(R)$, then there there exists a commutative diagram 
\begin{eqnarray*}
\xymatrix{
0 \ar[r] & M' \ar[r] \ar[d] & N' \ar[r] \ar[d] & L \ar[r] \ar@{=}[d] & 0 \\
0 \ar[r] & M \ar[r] & N \ar[r] & L \ar[r] & 0 \\
}
\end{eqnarray*}
in $\Mod_{\cC}(R)$ such that $(M',N') \in \Mod_\cC^{\iso}(R)^2$, the top row is strictly exact, and the bottom row is the given strict exact sequence. 
By Corollary \ref{isolated projective module} (ii), there exists a pair $(\cP_L,\pi_L)$ of a projective object $\cP_L$ of $\Mod_{\cC}(R)$ that is isolated at $0$ and a strict epimorphism $\pi_L \colon \cP_L \to L$. By the lifting property of projective objects, we have a commutative diagram of strictly exact sequences
\begin{eqnarray*}
\xymatrix{
0 \ar[r] & \ker(\pi_L) \ar[r] \ar[d] & \cP_L \ar[r]^{\pi_L} \ar[d] & L \ar[r] \ar@{=}[d] & 0 \\
0 \ar[r] & M \ar[r] & N \ar[r] & L \ar[r] & 0 \\
}
\end{eqnarray*}
and by Proposition \ref{Modiso is quasi-Abelian} (i) we have that $\ker(\pi_L)$ is isolated at $0$, concluding the proof.
\end{proof}

\begin{rmk}
Although we have not studied the question in depth, we think that it is probably false that the canonical functors $\Der^{\iso}_{\cC}(R) \to \Der_{\cC}(R)$ and $\Der^{\iso +}_{\cC}(R) \to \Der^+_{\cC}(R)$ are fully faithful.
\end{rmk}

Next, we study the relations between $\rCfin^{\iso}$ and $\wh{\otimes}_R$. For any $(M_0,M_1) \in \Modiso_{\cC}(R)^2$, we denote by $A^{\cD_3}_{M_0,M_1}$ the natural morphism $\rCfin(X,M_0) \wh{\otimes}_R M_1 \to \rCfin(X,M_0 \wh{\otimes}_R M_1)$ in $\Modiso_{\cC}(R)$.

\begin{thm}
\label{absorbing law}
Let $(M_0,M_1) \in \Modiso_{\cC}(R)^2$. Then, the following hold:
\begin{itemize}
\item[(i)] The natural $R$-linear homomorphism $\varphi \colon \rCfin(X,M_0) \otimes_R M_1 \to \rCfin(X,M_0 \wh{\otimes}_R M_1)$ is bijective.
\item[(ii)] If $\cC = \NBAb$, then $A^{\cD_3}_{M_0,M_1}$ is an isometric isomorphism.
\item[(iii)] If $R$ is isolated at $0$ and there exists a finite subset $S \subset M_1$ such that the natural morphism $\ell^{\cC}(S,R) \to M_1$ in $\Modiso_{\cC}(R)$ is a strict epimorphism, then $A^{\cD_3}_{M_0,M_1}$ is an isomorphism.
\item[(iv)] If $\cC = \BAb$, $X$ is $\Z$ equipped with the discrete topology, $R \neq \ens{0}$, and $R$ is isolated at $0$, then $A^{\cD_3}_{R,\ell^{\cC}((\Z,1_{\Z}),R)}$ is not an isomorphism (cf.\ Example \ref{trivially normed set} (ii)).
\end{itemize}
\end{thm}

\begin{proof}
We show the assertion (i). The surjectivity of the $R$-linear homomorphism $\varphi$ in the assertion (i) follows from Proposition \ref{idempotent generates C(X,R)} (i) and Proposition \ref{Modiso is monoidal subcategory} (i). We show the injectivity of $\varphi$. Let $f \in \rCfin(X,M_0) \otimes_R M_1$ with $\varphi(f) = 0$. Take an $(f_i,m_i)_{i=0}^{n} \in (\rCfin(X,M_0) \times M_1)^{n+1}$ with $n \in \N$ and $f = \sum_{i=0}^{n} f_i \otimes m_i$. By Proposition \ref{idempotent generates C(X,R)} (ii), there exists a finite subset $\cU \subset \CO(X) \setminus \ens{\emptyset}$ such that $\cU$ is a disjoint cover of $X$ and for any $(i,U) \in \N \times \cU$ with $i \leq n$, $f_i |_U$ is constant. For each $(i,U) \in \N \times \cU$ with $i \leq n$, we denote by $m_{i,U} \in M_0$ the value of $f_i |_U$, which makes sense by $U \neq \emptyset$. Since $\varphi(f) |_U$ is a constant map with value $\sum_{i=0}^{n} m_{i,U} \otimes m_i$, we have $\sum_{i=0}^{n} m_{i,U} \otimes m_i = 0$ for any $U \in \cU$ by $\varphi(f) = 0$. For each $U \in \cU$, we denote by $i_U$ the $R$-linear homomorphism $M_0 \to \rCfin(X,M_0)$ assigning $1_{X,R,U} m$ to each $m \in M_0$. We obtain
\begin{eqnarray*}
& & f = \sum_{i=0}^{n} f_i \otimes m_i = \sum_{i=0}^{n} \sum_{U \in \cU} i_U(m_{i,U}) \otimes m_i = \sum_{U \in \cU} \sum_{i=0}^{n} i_U(m_{i,U}) \otimes m_i \\
& = & \sum_{U \in \cU} (i_U \otimes_R \id_{M_1}) \left( \sum_{i=0}^{n} m_{i,U} \otimes m_i \right) = \sum_{U \in \cU} (i_U \otimes_R \id_{M_1})(0) = 0.
\end{eqnarray*}
We show the assertion (ii). If $X = \emptyset$, then both objects of $\Modiso_{\cC}(R)$ are zero objects, which have unique norms. Therefore, we may assume $X \neq \emptyset$. First, we show that $\n{\varphi(f)}_{\rCfin(X,M_0 \wh{\otimes}_R M_1)} < \n{f}_{\rCfin(X,M_0) \wh{\otimes}_R M_1} + \epsilon$ for any $(\epsilon,f) \in (0,\infty) \times (\rCfin(X,M_0) \wh{\otimes}_R M_1)$. By the definition of the tensor seminorm for $\cC = \NBAb$, there exists an $(f_i,m_i)_{i=0}^{n} \in (\rCfin(X,M_0) \times M_1)^{n+1}$ with $n \in \N$ and $f = \sum_{i=0}^{n} f_i \otimes m_i$ such that $\max_{i=0}^{n} \n{f_i}_{\rCfin(X,M_0)} \n{m_i}_{M_1} < \n{f}_{\rCfin(X,M_0) \wh{\otimes}_R M_1} + \epsilon$. Then, we have
\begin{eqnarray*}
& & \n{\varphi(f)(x)}_{M_0 \wh{\otimes}_R M_1} = \n{\sum_{i=0}^{n} f_i(x) \otimes m_i}_{M_0 \wh{\otimes}_R M_1} \leq \max{}_{i=0}^{n} \n{f_i(x)}_{M_0} \n{m_i}_{M_1} \\
& \leq & \max{}_{i=0}^{n} \n{f_i}_{\rCfin(X,M_0)} \n{m_i}_{M_1} < \n{f}_{\rCfin(X,M_0) \wh{\otimes}_R M_1} + \epsilon
\end{eqnarray*}
for any $x \in X$, and hence $\n{\varphi(f)}_{\rCfin(X,M_0 \wh{\otimes}_R M_1)} < \n{f}_{\rCfin(X,M_0) \wh{\otimes}_R M_1} + \epsilon$. This implies that $\varphi$ is a contraction.

\vspace{0.1in}
Next, we show that $\n{\varphi^{-1}(f)}_{\rCfin(X,M_0) \wh{\otimes}_R M_1} < \n{f}_{\rCfin(X,M_0 \wh{\otimes}_R M_1)} + \epsilon$ for any $(\epsilon,f) \in (0,\infty) \times \rCfin(X,M_0 \wh{\otimes}_R M_1)$. By Proposition \ref{idempotent generates C(X,R)} (ii), there exists a pair $(\cU,e)$ of a finite subset $\cU \subset \CO(X) \setminus \ens{\emptyset}$ and a map $e \colon \cU \to M_0 \wh{\otimes}_R M_1$ such that $\cU$ is a disjoint cover of $X$ and $\sum_{U \in \cU} 1_{X,R,U} e(U) = f$. Then, we have $\n{f}_{\rCfin(X,M_0 \wh{\otimes}_R M_1)} = \sup_{U \in \cU} \n{e(U)}_{M_0 \wh{\otimes}_R M_1}$.

\vspace{0.1in}
For any $U \in \cU$, there exists an $(m_{U,0,i},m_{U,1,i})_{i=0}^{n_U} \in (M_0 \times M_1)^{n_U + 1}$ with $n_U \in \N$ such that $\sum_{i=0}^{n_U} m_{U,0,i} \otimes m_{U,1,i} = e(U)$ by Proposition \ref{Modiso is monoidal subcategory} (i). By the definition of the tensor seminorm for $\cC = \NBAb$, we choose $n_U$ and $(m_{U,0,i},m_{U,1,i})_{i=0}^{n_U}$ so that the inequality $\max_{i=0}^{n_U} \n{m_{U,0,i}}_{M_0} \n{m_{U,1,i}}_{M_1} < \n{e(U)}_{M_0 \wh{\otimes}_R M_1} + \epsilon$ holds. Then, we have
\begin{eqnarray*}
& & \n{\varphi^{-1}(f)}_{\rCfin(X,M_0) \wh{\otimes}_R M_1} = \n{\varphi^{-1} \left( \sum_{U \in \cU} 1_{X,R,U} e(U) \right) }_{\rCfin(X,M_0) \wh{\otimes}_R M_1} \\
& = & \n{\varphi^{-1} \left( \sum_{U \in \cU} 1_{X,R,U} \sum_{i=0}^{n_U} m_{U,0,i} \otimes m_{U,1,i} \right)}_{\rCfin(X,M_0) \wh{\otimes}_R M_1} \\
& = & \n{\sum_{U \in \cU} \sum_{i=0}^{n_U} (1_{X,R,U} m_{U,0,i}) \otimes m_{U,1,i}}_{\rCfin(X,M_0) \wh{\otimes}_R M_1} \\
& \leq & \sup_{U \in \cU} \textstyle\max_{i=0}^{n_U}\displaystyle \n{1_{X,R,U} m_{U,0,i}}_{\rCfin(X,M_0)} \n{m_{U,1,i}}_{M_1} \\
& = & \sup_{U \in \cU} \textstyle\max_{i=0}^{n_U}\displaystyle \n{m_{U,0,i}}_{M_0} \n{m_{U,1,i}}_{M_1} < \sup_{U \in \cU} \n{e(U)}_{M_0 \wh{\otimes}_R M_1} + \epsilon \\
& < & \n{f}_{\rCfin(X,M_0 \wh{\otimes}_R M_1)} + \epsilon
\end{eqnarray*}
by $\cC = \NBAb$. This implies that $\varphi^{-1}$ is a contraction.

\vspace{0.1in}
We show the assertion (iii). We may assume $0 \notin S$. We denote by $\pi$ the natural morphism $\ell^{\cC}(S,R) \to M$ in $\Modiso_{\cC}(R)$. We denote by $E$ the null sequence
\begin{eqnarray*}
\ker(\pi) \hookrightarrow \ell^{\cC}(S,R) \stackrel{\pi}{\to} M \to 0
\end{eqnarray*}
of $\Modiso_{\cC}(R)$, which is strictly coexact by the assumption. We denote by $\varphi$ the natural morphism $\rCfin(X,R) \wh{\otimes}_R E \to \rCfin(X,E)$ of null sequences of $\Modiso_{\cC}(R)$, which is termwise bijective by the assertion (i). By Proposition \ref{Cfin is strongly exact}, $\rCfin(X,E)$ is strictly coexact. By the strong right exactness of $\rCfin(X,R) \wh{\otimes}_R (\cdot)$, $\rCfin(X,R) \wh{\otimes}_R E$ is strictly coexact. For each $r \in (0,\infty)$, we denote by $R r$ the regular $R$-module object rescaled by $r$. We consider the commutative diagram
\begin{eqnarray*}
\xymatrix{
\rCfin(X,R) \wh{\otimes}_R \ell^{\cC}(S,R) \ar@{=}[d] & \\
\rCfin(X,R) \wh{\otimes}_R \bigoplus_{m \in S} R \n{m}_M \ar[d] & \bigoplus_{m \in S} \rCfin(X,R) \wh{\otimes}_R R \n{m}_M \ar[l] \ar[d] \\
\rCfin \left(X, \bigoplus_{m \in S} R \n{m}_M \right) & \bigoplus_{m \in S} \rCfin(X,R \n{m}_M) \ar[l] \\
\rCfin(X,\ell^{\cC}(S,R)) \ar@{=}[u] & \\
}
\end{eqnarray*}
of $\Mod_{\cC}(R)$.The horizontal arrows are isomorphisms by the additivity of the functors and the right vertical arrow is an isomorphism by the definition of rescaling. This implies that the left vertical arrow is an isomorphism, and hence $\varphi$ is an isomorphism at the second term. Since $\varphi$ is a termwise bijective morphism between strictly coexact sequences and is an isomorphism at the second term, it is an isomorphism at the first term. By the universality of cokernels, $\varphi$ is an isomorphism at the third term, i.e.\ $A^{\cD_3}_{M_0,M_1}$ is an isomorphism.

\vspace{0.1in}
We show the assertion (iv). For this purpose, it suffices to show that the inverse of the underlying $R$-linear homomorphism $\psi$ of $A^{\cD_3}_{R,\ell^{\cC}((\Z,1_{\Z}),R)}$ is unbounded. We denote by $\iota$ the canonical embedding $(\Z,1_{\Z}) \hookrightarrow \ell^{\cC}((\Z,1_{\Z}),R)$. Since $R$ is isolated at $0$, there exists an $\epsilon \in (0,\infty)$ such that $\set{r \in R}{\n{r} < \epsilon} = \ens{0}$. Since $R \neq \ens{0}$, we have $\n{1}_R \neq 0$. It suffices to show that $\psi^{-1}$ is of operator norm $\geq \n{1}_R^{-1} \epsilon^2 n$ for any $n \in \N$. Put $f_n \coloneqq \sum_{i=0}^{n} 1_{X,R,\ens{n}} \otimes \iota(n) \in \rCfin(X,R) \wh{\otimes}_R \ell^{\cC}((\Z,1_{\Z}),R)$. Since $\psi(f_n)$ assigns $\iota(i)$ to each $i \in X$ with $0 \leq i \leq n$ and $0$ to each $i \in X$ with $i < 0$ or $i > n$, we have
\begin{eqnarray*}
\n{\psi(f_n)}_{\rCfin(X,\ell^{\cC}((\Z,1_{\Z}),R))} = \textstyle\max_{i=0}^{n}\displaystyle \n{\iota(i)}_{\ell^{\cC}((\Z,1_{\Z}),R)} = \textstyle\max_{i=0}^{n}\displaystyle \n{1}_R 1_{\Z}(i) = \n{1}_R \neq 0.
\end{eqnarray*}
For any $(g_h,a_h)_{h=0}^{m} \in (\rCfin(X,R),\ell^{\cC}((\Z,1_{\Z}),R))^{m+1}$ with $m \in \N$ and $\sum_{h=0}^{m} g_h \otimes a_h = f_n$, we have $m \geq n$ by the computation of the rank of the $(n \times n)$-matrix corresponding to $\sum_{h=0}^{m} (g_h(i))_{i=0}^{n} \otimes (a_h(j))_{j=0}^{n} \in \Z^{n+1} \otimes_{\Z} \Z^{n+1}$. This implies $\n{f_n}_{\rCfin(X,R) \wh{\otimes}_R \ell^{\cC}((\Z,1_{\Z}),R)} \geq \epsilon^2 n$, because every non-zero element of $\rCfin(X,R) \wh{\otimes}_R \ell^{\cC}((\Z,1_{\Z}),R)$ is of norm $\geq \epsilon^2$. This implies that $\psi^{-1}$ is of operator norm $\geq \n{1}_R^{-1} \epsilon^2 n$.
\end{proof}

\begin{rmk}
Theorem \ref{absorbing law} (i) just implies that $A^{\cD_3}_{M_0,M_1}$ is bijective but does not imply that it is an isomorphism, as the open mapping theorem does not necessarily hold for Banach $R$-modules. Indeed, although it is an isometric isomorphism when $\cC = \NBAb$ by Theorem \ref{absorbing law} (ii), it is not necessarily an isomorphism when $\cC = \BAb$ by Theorem \ref{absorbing law} (iv). The condition on $M_1$ in Theorem \ref{absorbing law} (iii) is not equivalent to the condition that $M_1$ is finitely generated as an $R$-module. Indeed, if $R = \Z_{\infty}$ and $M_1 = \Z_{\triv}$ (cf.\ Example \ref{example of Banach ring isolated at 0} (i) and (iii)), then $M_1$ is finitely generated but does not satisfy the condition. It is also remarkable that the bijectivity is a special feature of the isolated setting and the local field setting (cf.\ \cite{Mih14} Theorem 4.12). For example, when $\cC = \BAb$, $X$ is the residue field of $\Cp$, $R = M_0 = \Z_{\infty}$, and $M_1 = \Cp$, then the morphism $\rCbd(X,M_0) \wh{\otimes}_R M_1 \to \rCbd(X,M_0 \wh{\otimes}_R M_1)$ in $\Mod_{\cC}(\Z_\infty)$ analogous to $A^{\cD_3}_{M_0,M_1}$ is not surjective. Indeed, every function $X \to M_0 \wh{\otimes}_R M_1 \cong \Cp$ in the image of $A^{\cD_3}_{M_0,M_1}$ can be approximated by functions whose images are contained in finitely generated subfields of $\Cp$, but the \Teichmuller embedding $X \hookrightarrow \Cp$ cannot be approximated by such functions.
\end{rmk}

As an immediate consequence of Proposition \ref{isolated at 0 implies discrete} and Theorem \ref{absorbing law} (iii), we obtain the following corollary.

\begin{crl}
\label{Base change maximal points}
Suppose that $R$ is isolated at $0$. Let $M \in \Modiso_{\cC}(R)$. For any ideal $I \subset R$, $R/I$ is an object of $\Modiso_{\cC}(R)$ with respect to the quotient seminorm, and $A^{\cD_3}_{M,R/I}$ is an isomorphism. In particular, for any $x \in \cM(R)$, if $\set{a \in R}{x(a) = 0}$ is a maximal ideal, then $A^{\cD_3}_{M,\kappa(x)}$ is an isomorphism, where $\kappa(x)$ denotes the completed residue field at $x$.
\end{crl}

\section{Schauder bases and $\Z$-forms}
\label{Schauder bases and Z-forms}

In this last section, we study some further properties of the algebras $\rCfin(X, R)$ for any topological space $X$ and specific choices of the Banach ring $R$. We show in Proposition \ref{Cfin is free} that $\rC(\zeta(X),R)$ is free as an $R$-module if $R$ is isolated at $0$, and topologically free as a Banach $R$-module if $R$ is non-Archimedean by constructing (algebraic and topological) bases in both cases consisting of characteristic functions of clopen subsets. This extends well-known results about the existence of an orthonormal Schauder basis of $\rC(\zeta(X),R)$ when $R$ is a complete non-Archimedean valued field (cf.\ \cite{Mon70} IV 3 Corollaire 1, \cite{BGR84} 2.5.2 Lemma 2, and the proof of \cite{Sch02} Proposition 10.1). We also show in Example \ref{Mahler basis} in the specific case $X = \Z_p$ that the algebra $\rCfin(\Z_p, R)$ has another orthonormal Schauder basis, called the Mahler basis, whose properties under base change are very different from the ones of the orthonormal Schauder basis coming from the characteristic functions of clopen subsets. We conclude by proving the Stone--Weierstrass Theorem for $\rCfin(X, R)$, under a suitable hypothesis on $R$ in Theorem \ref{non-Archimedean Stone--Weierstrass}.

\vspace{0.1in}
Let $M \in \Mod_{\cC}(R)$. A subset $S \subset M$ is said to be a {\it $\cC$-orthonormal Schauder $R$-linear basis of $M$} if $S$ consists of elements of norm $1$ and the natural $R$-linear homomorphism $R^{\oplus S} \to M$ extends to an isometric $R$-linear isomorphism $\ell^{\cC}(S,R) \to M$. We recall that for a $U \in \CO(X)$, $\cl_{\zeta,X}(U)$ denotes the unique clopen subset of $\zeta(X)$ satisfying $\iota_{\zeta,X}^{-1}(\cl_{\zeta,X}(U)) = U$ (cf.\ Corollary \ref{extension property for clopen subset}). A monoid object $A$ of $\NBAb$ is said to be {\it with submultiplicative norm} (in the sense of \cite{BM21} p.\ 6) if the canonical morphisms $\Z_{\triv} \to A$ and $A \wh{\otimes} A \to A$ defining the monoid structure are contractions, or equivalently, $A$ is a monoid object of the symmetric subcategory of $\NBAb$ with the same class of objects whose class of morphisms consists of contracting morphisms. As a consequence of Theorem \ref{absorbing law} (i), we obtain a generalisation of \cite{Nob68} Satz 1 (cf.\ \cite{Sch19} Theorem 5.4).

\begin{prp}
\label{Cfin is free}
There exists an $\cF \subset \CO(X)$ satisfying the following:
\begin{itemize}
\item[(i)] For any monoid object $A$ of $\BAb$ isolated at $0$, i.e.\ a Banach ring in the sense in \cite{BM21} p.\ 6 isolated at $0$, $\rCfin(X,A)$ is a free $A$-module with $A$-linear basis $\set{1_{X,A,U}}{U \in \cF}$.
\item[(ii)] For any monoid object $A$ of $\NBAb$ with submultiplicative norm isolated at $0$, the $A$-linear basis $\set{1_{X,A,U}}{U \in \cF}$ of $\rCfin(X,A)$ is an $\NBAb$-orthonormal Schauder $A$-linear basis of $\rCfin(X,A)$.
\item[(iii)] For any complete non-Archimedean valued ring (resp.\ field) $A$, $\set{1_{\zeta(X),A,V}}{V \in \cF'}$ is an $\NBAb$-orthonormal Schauder $A$-linear basis of $\rC(\zeta(X),A)$, where $\cF'$ denotes the image of $\cF$ by $\cl_{\zeta,X}$.
\end{itemize}
\end{prp}

\begin{proof}
We show the assertion (i). When $X = \emptyset$, then $\cF$ can be taken as $\emptyset$. Therefore, we may assume $X \neq \emptyset$. By $X \neq \emptyset$ and \cite{Nob68} Satz 1, there exists an $\cF \subset \CO(X)$ such that $\set{1_{X,\Z_{\infty},U}}{U \in \cF}$ is a $\Z$-linear basis of $\rCfin(X,\Z_{\infty})$. Since $\Z_{\infty}$ is an initial object in the category of monoids of $\BAb$ and monoid homomorphisms, $\set{1_{X,A,U}}{U \in \cF}$ is an $A$-linear basis of $\rCfin(X,A)$ for any monoid $A$ of $\BAb$ by Theorem \ref{absorbing law} (i) applied to $(R,M_0,M_1) = (\Z_{\infty},\Z_{\infty},A)$.

\vspace{0.1in}
We show the assertion (ii). By the assertion (i), there exists an $\cF \subset \CO(X)$ such that $\set{1_{X,\Z_{\triv},U}}{U \in \cF}$ is a $\Z$-linear basis of $\rCfin(X,\Z_{\triv})$. Since $\n{\cdot}_{\rCfin(X,\Z_{\triv})}$ is the trivial norm, any $\Z$-linear basis of $\rCfin(X,\Z_{\triv})$ is an $\NBAb$-orthonormal Schauder $\Z_{\triv}$-linear basis. By Theorem \ref{absorbing law} (ii), the canonical morphism $\cA^{\cD_3}_{\Z_{\triv}, A} \colon \rCfin(X,\Z_{\triv}) \wh{\otimes}_{\Z_{\triv}} A \to \rCfin(X,\Z_{\triv} \wh{\otimes}_{\Z_{\triv}} A)$ is an isometric isomorphism. Since the unitor $\Z_{\triv} \wh{\otimes}_{\Z_{\triv}} A \to A$ is an isometry by the assumption of the submultiplicativity, the $A$-linear basis $\set{1_{X,A,U}}{U \in \cF}$ of $\rCfin(X,A)$ is an $\NBAb$-orthonormal Schauder $A$-linear basis.

\vspace{0.1in}
We show the assertion (iii). Let $k$ be a complete non-Archimedean valued field. Put $O_k \coloneqq \set{c \in k}{\v{c} \leq 1}$, $m_k \coloneqq \set{c \in O_k}{\v{c} < 1}$, and $\kappa_k \coloneqq O_k/m_k$. By the assertion (i), $\set{1_{X,\kappa_k,U}}{U \in \cF}$ forms a $\kappa_k$-linear basis of $\rCfin(X,\kappa_k)$. Since $\kappa_k$ is isolated at $0$, $\set{1_{\zeta(X),\kappa_k,\cl_{\zeta,X}(U)}}{U \in \cF}$ forms a $\kappa_k$-linear basis of $\rC(\zeta(X),\kappa_k)$ by Proposition \ref{extension property} and Corollary \ref{extension property for clopen subset}. For any $U_0 \in \CO(X)$, $1_{X,k,U_0}$ belongs to the Abelian subgroup of $k^X$ generated by $\set{1_{X,k,U}}{U \in \cF}$ by the definition of $\cF$. Therefore, for any $\ol{U}_0 \in \CO(\zeta(X))$, $1_{\zeta(X),k,\ol{U}_0}$ belongs to the Abelian subgroup of $k^{\zeta(X)}$ generated by $\set{1_{\zeta(X),k,\cl_{\zeta,X}(U)}}{U \in \cF}$ by Corollary \ref{extension property for clopen subset}. This implies that $\set{1_{\zeta(X),k,\cl_{\zeta,X}(U)}}{U \in \cF}$ generates the Abelian subgroup of $\rC(\zeta(X),k)$ generated by $\set{1_{\zeta(X),k,\ol{U}}}{\ol{U} \in \CO(\zeta(X))}$ again by Corollary \ref{extension property for clopen subset}, and hence generates a dense $k$-subalgebra of $\rC(\zeta(X),k)$ by \cite{Mih14} Theorem 4.12.

\vspace{0.1in}
Therefore, it suffices to show $\n{\sum_{U \in \cF} a_U 1_{\zeta(X),k,\cl_{\zeta,X}(U)}}_{\rC(\zeta(X),k)} = \sup_{U \in \cF} \v{a_U}$ for any $(a_U)_{U \in \cF} \in k^{\oplus \cF}$. By the finiteness of $\set{a_U}{U \in \cF} \subset k$, there exists a $U_0 \in \cF$ such that $\v{a_U} \leq \v{a_{U_0}}$ for any $U \in \cF$. If $a_{U_0} = 0$, then both hand sides are $0$. Therefore, we may assume $a_{U_0} \neq 0$. We have $(a_{U_0}^{-1} a_U)_{U \in \cF} \in O_k^{\oplus \cF} \setminus m_k^{\oplus \cF}$, and hence $((a_{U_0}^{-1} a_U) + m_k)_{U \in \cF}$ is a non-zero element of $\kappa_k^{\oplus \cF}$. Since $\set{1_{\zeta(X),\kappa_k,\cl_{\zeta,X}(U)}}{U \in \cF}$ is a $\kappa_k$-linear basis of $\rC(\zeta(X),\kappa_k)$, we have $\sum_{U \in \cF} ((a_{U_0}^{-1} a_U) + m_k) 1_{\zeta(X),\kappa_k,\cl{\zeta,X}(U)} \neq 0$. Therefore, there exists an $x \in \zeta(X)$ such that $\sum_{U \in \cF} ((a_{U_0}^{-1} a_U) + m_k) 1_{\zeta(X),\kappa_k,\cl{\zeta,X}(U)}(x) \neq 0$, or equivalently, $\sum_{U \in \cF \cap x} ((a_{U_0}^{-1} a_U) + m_k) \neq 0$ by the definition of $\zeta(X)$ as the Stone space of $\CO(X)$. This implies $\sum_{U \in \cF \cap x} a_{U_0}^{-1} a_U \in O_k \setminus m_k$, and hence $\v{\sum_{U \in \cF \cap x} a_U} = \v{a_{U_0}}$. We obtain
\begin{eqnarray*}
& & \n{\sum_{U \in \cF} a_U 1_{\zeta(X),k,\cl_{\zeta,X}(U)}}_{\rC(\zeta(X),k)} \geq \v{\sum_{U \in \cF} a_U 1_{\zeta(X),k,\cl_{\zeta,X}(U)}(x)} = \v{\sum_{U \in \cF \cap x} a_U} \\
& = & \v{a_{U_0}} = \sup_{U \in \cF} \v{a_U}.
\end{eqnarray*}
\end{proof}

Suppose that $R$ is isolated at $0$. For an $A \in \Alg_{\cC}(R)$, we denote by $B^{\cD_3}_A$ the composition $\rCfin(X,R) \wh{\otimes}_R A \to \rC(\zeta(X),A)$ of the inverse of the restriction map $\rC(\zeta(X),R) \wh{\otimes}_R A \to \rCfin(X,R) \wh{\otimes}_R A$, which is an isometric isomorphism in $\Alg_{\cC}(A)$ by Proposition \ref{extension property}, and the natural morphism $\rC(\zeta(X),R) \wh{\otimes}_R A \to \rC(\zeta(X),A)$ in $\Alg_{\cC}(A)$.

\begin{crl}
\label{existence of Z-form}
Let $R$ be a monoid object of $\NBAb$ with submultiplicative norm isolated at $0$ or a complete non-Archimedean valued ring (resp.\ field). Then $B^{(\NBAb,X,\Z_{\triv})}_R$ is an isometric $R$-linear isomorphism.
\end{crl}

\begin{proof}
The claim for the isolated case follows from Proposition \ref{extension property} and Proposition \ref{Cfin is free} (ii), because $B^{(\NBAb,X,\Z_{\triv})}_R$ coincides with the composite of $A^{\cD_3}_{\Z_{\triv},R}$ and the inverse of the restriction map $\rC(\zeta(X),R) \to \rCfin(X,R)$. Suppose that $R$ is a complete valued ring (resp.\ field). By Proposition \ref{Cfin is free} (iii), there exists an $\cF \subset \CO(X)$ such that $\set{1_{X,\Z_{\triv},U}}{U \in \cF}$ is a $\Z$-linear basis of $\rCfin(X,\Z_{\triv})$ and $\set{1_{\zeta(X),R,\cl_{\zeta,X}(U)}}{U \in \cF}$ is an $\NBAb$-orthonormal Schauder $R$-linear basis of $\rC(\zeta(X),R)$. Since $\n{\cdot}_{\rCfin(X,\Z_{\triv})}$ is the trivial norm, any $\Z$-linear basis of $\rCfin(X,\Z_{\triv})$ is an $\NBAb$-orthonormal Schauder $\Z_{\triv}$-linear basis. Therefore, the natural $\Z_{\triv}$-linear homomorphism $\Z_{\triv}^{\oplus \cF} \to \rCfin(X,\Z_{\triv})$ gives an isometric $\Z_{\triv}$-linear isomorphism $\ell^{\NBAb}(\cF,\Z_{\triv}) \to \rCfin(X,\Z_{\triv})$, and $\set{1_{X,\Z_{\triv},U} \otimes 1}{U \in \cF}$ is an $\NBAb$-orthonormal Schauder $R$-linear basis of $\rCfin(X,\Z_{\triv}) \wh{\otimes}_{\Z_{\triv}} R$ by Proposition \ref{isolated free module} (ii) applied to $(\NBAb,\Z_{\triv},R)$. Since $B^{(\NBAb,X,\Z_{\triv})}_R$ sends the $\NBAb$-orthonormal Schauder $R$-linear basis $\set{1_{X,\Z_{\triv},U} \otimes 1}{U \in \cF}$ to the $\NBAb$-orthonormal Schauder $R$-linear basis $\set{1_{\zeta(X),R,\cl_{\zeta,X}(U)}}{U \in \cF}$, it is an isometric $R$-linear isomorphism.
\end{proof}

The existence of an $\NBAb$-orthonormal Schauder $R$-linear basis for $\rCfin(X, R)$ has the following remarkable consequence:

\begin{crl}
\label{Schauder basis implies strongly flat}
Let $R$ be a monoid object of $\NBAb$ isolated at $0$. Then, for all topological spaces $X$, the Banach $R$-module $\rCfin(X, R)$ is strongly flat as an object of $\Mod_{\NBAb}(R)$. 
\end{crl}

\begin{proof}
We may assume that the canonical morphisms $\Z_{\triv} \to R$ and $R \wh{\otimes} R \to R$ are contractions by a standard argument (cf.\ \cite{BM21} p.\ 6). Therefore, by Proposition \ref{Cfin is free} (ii), $\rCfin(X,R)$ admits an $\NBAb$-orthonormal Schauder $R$-linear basis, and hence we have that $\rCfin(X, R) \cong \ell^{\NBAb}((S,1_S), R)$ for a set $S$ (cf.\ Example \ref{trivially normed set} (ii)). This implies that $\rCfin(X, R)$ is a projective object of $\Mod_{\NBAb}(R)$ and projective objects are strongly flat by \cite{BK20} Proposition 3.11.
\end{proof}

Let us fix two monoid objects $R$ and $R_0$ of $\cC$ and a morphism $R_0 \to R$ of monoids. For an $A \in \Alg_{\cC}(R)$, a {\it $(\cC,R_0)$-form of $A$ over $R_0$} is an object $A_{R_0} \in \Alg_{\cC}(R_0)$ equipped with an isomorphism $\varphi_A \colon A_{R_0} \wh{\otimes}_{R_0} R \to A$ in $\Alg_{\cC}(R)$. The most interesting case is when $(\cC,R_0) = (\BAb,\Z_\infty)$ or $(\NBAb,\Z_{\triv})$, in which case the $(\cC,R_0)$-forms of monoid objects of $\cC$ have an absolute nature due to the initiality of $R_0$. Indeed, the symmetric monoidal structure $\wh{\otimes}$ of $\cC$ is naturally identified with $\wh{\otimes}_{\Z_{\infty}}$ when $\cC = \BAb$, and with $\wh{\otimes}_{\Z_{\triv}}$ when $\cC = \NBAb$. We will talk generically of {\it $\Z$-forms} of a Banach algebra when both $(\BAb,\Z_{\infty})$ or $(\NBAb,\Z_{\triv})$ are considered or allowed.

\begin{exm}
\label{examples Z-forms}
\begin{itemize}
\item[(i)] By the proof of Corollary \ref{existence of Z-form}, we know that $\rCfin(X,\Z_{\triv})$ is a canonical $(\NBAb,\Z_{\triv})$-form of $\rC(\zeta(X),k)$ over $k$ for any complete non-Archimedean valued field $k$. 
\item[(ii)] On the contrary, $\rCfin(X,\Z_{\infty})$ is not necessarily a $(\BAb,\Z_{\infty})$-form of $\rCfin(X,\R)$ of $\rCfin(X,\C)$ if $X$ is a connected compact Hausdorff topological space with at least two points. For example, the image of $\rCfin(X,\Z_{\infty}) \wh{\otimes}_{\Z_\infty} \R$ is a proper subalgebra of $\rCfin(X,\R)$.
\item[(iii)] The $\Z$-forms of a fixed Banach algebra $A \in \Alg_{\cC}(R)$ are not necessarily unique. Later on in this section we will present an important example of a Banach algebra admitting two non-isomorphic $\Z$-forms. These amounts to two different Schauder bases for the Banach algebra $\rC(\zeta(X),R)$. See Example \ref{Mahler basis} for details.
\end{itemize}
\end{exm}

It is important to recognise when a Banach algebra admits a $\Z$-form. We show that the notion of a $\Z$-form is essentially equivalent to that of an orthonormal Schauder basis whose multiplication table is a matrix with coefficients in $\Z$.

\begin{prp}
\label{construction of orthonormal Schauder basis}
Let $A \in \Alg_{\NBAb}(R)$.
\begin{itemize}
\item[(i)] Let $A_{\Z}$ be an $(\NBAb,\Z_{\triv})$-form of $A$ over $R$ such that $\n{\cdot}_{A_{\Z}}$ is the trivial norm. For any $\Z$-linear basis $S$ of $A_{\Z}$, $\varphi_A(S)$ is an $\NBAb$-orthonormal Schauder $R$-linear basis of $A$ whose multiplication table is a matrix with coefficients in $\Z$.
\item[(ii)] Let $S_R \subset A$ be an $\NBAb$-orthonormal Schauder $R$-linear basis of $A$ such that the $\Z$-submodule $\Z S_R \subset A$ generated by $S_R$ is a subring. Then, $(\Z S_R)_{\triv}$ is an $(\NBAb,\Z_{\triv})$-form of $A$ over $R$ with respect to the morphism $(\Z S_R)_{\triv} \wh{\otimes}_{\Z_{\triv}} R \to A$ in $\Alg_{\NBAb}(R)$ induced by the inclusion $\Z S_R \hookrightarrow A$.
\end{itemize}
\end{prp}

\begin{proof}
The assertion (i) immediately follows from Proposition \ref{isolated free module} (ii), because the natural morphism $\ell^{\NBAb}((S,1_S),\Z_{\triv}) \to A_{\Z}$ is an isometric isomorphism (cf.\ Example \ref{trivially normed set} (ii)). The assertion (ii) also immediately follows from Proposition \ref{isolated free module} (ii), because the natural morphism $\ell^{\NBAb}(S_R,(\Z_R)_{\triv}) \to (\Z S_R)_{\triv}$ is an isometric isomorphism, where $\Z_R$ denote the set theoretic image of the canonical ring homomorphism $\Z \to R$.
\end{proof}

As an immediate consequence of Corollary \ref{existence of Z-form} and Proposition \ref{construction of orthonormal Schauder basis} (i), we obtain a partial generalisation of Proposition \ref{Cfin is free}.

\begin{crl}
\label{construction of orthonormal Schauder basis for C(zeta(X),k)}
Let $S$ be a $\Z$-linear basis of $\rCfin(X,\Z_{\triv})$, and $R$ a monoid object of $\NBAb$ with submultiplicative norm isolated at $0$ or a complete non-Archimedean valued ring (resp.\ field). Then $\set{f_R}{f \in S}$ is an $\NBAb$-orthonormal Schauder $R$-linear basis of $\rC(\zeta(X),R)$ whose multiplication table is a matrix with coefficients in $\Z$, where $f_R$ denotes the composite of the unique extension $\zeta(X) \to \Z$ of $f$ (cf.\ Proposition \ref{extension property}) and the canonical ring homomorphism $\Z \to R$.
\end{crl}

By Proposition \ref{Cfin is closed} (ii) and Corollary \ref{construction of orthonormal Schauder basis for C(zeta(X),k)}, we obtain the following:

\begin{crl}
\label{construction of orthonormal Schauder basis for C(X,k)}
Suppose $X \in \TDCH$. Let $S$ be a $\Z$-linear basis of $\rC(X,\Z_{\triv})$. Let $R$ be a monoid object $R$ of $\NBAb$ with submultiplicative norm isolated at $0$ or a complete non-Archimedean valued ring (resp.\ field). Then $\set{f_R}{f \in S}$ is an $\NBAb$-orthonormal Schauder $R$-linear basis of $\rC(X,R)$ whose multiplication table is a matrix with coefficients in $\Z$, where $f_R$ denotes the composite of $f$ and the canonical ring homomorphism $\Z \to R$.
\end{crl}

One of the benefits of choosing a $\Z$-form is that analysis of $A$ is reduced to a combinatorial computation based on a $\Z$-linear basis of $A_{\Z}$. Indeed, Corollary \ref{construction of orthonormal Schauder basis for C(X,k)} ensures that non-Archimedean functional analysis on $X$ is completely reduced to the multiplication table for a collection $S$ of functions $X \to \Z_{\triv}$ if $X \in \TDCH$.

\vspace{0.1in}
By Proposition \ref{Cfin is free}, we can choose $S$ to be a set of idempotents, in which case the multiplication table will be simpler, as the product just corresponds to the intersection of the corresponding clopen subsets. We say that $S$ is a {\it generalised van der Put basis of $\rCfin(X,\Z_{\triv})$} if it is a $\Z$-linear basis of $\rCfin(X,\Z_{\triv})$ such that $e_0 e_1 \in \ens{e_0,e_1,0}$ for any $(e_0,e_1) \in S^2$. Since $\rCfin(X,\Z_{\triv})$ is reduced, a generalised van der Put basis of $\rCfin(X,\Z_{\triv})$ consists of idempotents, and the multiplication table for a generalised van der Put basis is even simpler than a generic multiplication table by intersections of clopen subsets, because it is reduced to the monoid structure of $S \cup \ens{0}$ or the partial ordering of clopen subsets with respect to the inclusion. We give a criterion for the existence of a generalised van der Put basis.

\begin{prp}
\label{existence of generalised van der Put basis}
If $X$ is ultrametrisable, then $\rCfin(X,\Z_{\triv})$ admits a generalised van der Put basis.
\end{prp}

\begin{proof}
Fix an ultrametric $d_X$ on $X$ compatible with its topology. For each $(x,r) \in X \times (0,\infty)$, we denote by $\chi_{x,r}$ the characteristic function $X \to \ens{0,1}$ of $\set{y \in X}{d_X(x,y) < r} \in \CO(X)$. The assertion immediately follows from the proof of \cite{Sch19} Theorem 5.4 applied to the embedding $X \hookrightarrow \ens{0,1}^{X \times (0,\infty)}$ given by $(\chi_{x,r})_{(x,r) \in X \times (0,\infty)}$.
\end{proof}

A typical example of a generalised van der Put basis is the van der Put basis of $\rC(\Zp,\Z_{\triv})$ (cf.\ p.\ 263 \cite{Mah73} Chapter 16 \S 9 Theorem 4), as the name indicates. It is the set of the idempotents corresponding to clopen subsets of $\Zp$ of the form
\begin{eqnarray*}
\set{x \in \Zp}{\v{x-n}_p < n^{-1}}
\end{eqnarray*}
with $n \in \N \setminus \ens{0}$, where the $p$-adic valuation on $\Zp$ is normalised by $\v{p}_p = p^{-1}$. There are many classical results on $\rC(\Zp,\Qp)$ based on the presentation by the van der Put basis (cf.\ \cite{Sch84} \S 62). We now show that $\rC(\Zp,\Qp)$ admits two non-isomorphic $\Z$-forms, providing the example mentioned in Example \ref{examples Z-forms} (iii).

\begin{exm}
\label{Mahler basis}
The set of binomial polynomials is an $\NBAb$-orthonormal $\Qp$-linear basis of $\rC(\Zp,\Qp)$, by \cite{Mah58} Theorem 1, called the {\it Mahler basis}. This basis generates a $\Z$-submodule $\Int(\Z_{\triv}) \subset \rC(\Zp,\Qp)$ which is a subring and a $\Z$-form of $\rC(\Zp,\Qp)$ over $\Qp$, that is isomorphic to the so-called \emph{ring of integrally valued polynomials}. It is easy to check that $\Int(\Z_{\triv})$ is not contained in $\rCfin(\Zp,\Z)$, but is contained in
\begin{eqnarray*}
\set{f(T) \in \Q[T]}{\forall x \in \Z, f(x) \in \Z}
\end{eqnarray*}
under the identification of a polynomial and the corresponding polynomial function. Although the Mahler basis does not consist of idempotent elements, it has plenty of good combinatorial properties (cf.\ \cite{Mah73} \S 9 -- \S 10). For example, the Mahler basis is the dual basis of the topological basis $(([1] - [0])^n)_{n \in \N} \in \Zp \dbrack{\Zp}$ associated to the topological basis $(T^n)_{n \in \N} \in \Zp \dbrack{T} \cong \rH^0(\Gm,\cO_{\Gm})$ through the Iwasawa isomorphism $\Zp \dbrack{\Zp} \cong \Zp \dbrack{T}$, where $\Gm$ denotes the commutative formal group scheme over $\Spf(\Zp)$ given as the formal completion of the multiplicative group $\Spec(\Zp[T,(1+T)^{-1}])$ at the unit. Indeed, let $(n,i) \in \N^2$. If $i < n$, then we have $\sum_{j=0}^{i} (-1)^{n-j} \binom{i}{j} \binom{j}{n} = 0$. If $i \geq n$, then we have
\begin{eqnarray*}
& & \sum_{j=0}^{i} (-1)^{n-j} \binom{i}{j} \binom{j}{n} = \sum_{j=n}^{i} (-1)^{n-j} \binom{i}{j} \binom{j}{n} = \sum_{j=n}^{i} (-1)^{n-j} \frac{i!}{(i-j)! (j-n)! n!} \\
& = & \binom{i}{n} \sum_{j=n}^{i} (-1)^{n-j} \binom{i-n}{j-n} = \binom{i}{n} \sum_{j=0}^{i-n} (-1)^j \binom{i-n}{j} = \binom{i}{n} (1-1)^{i-n} \\
& = & 
\left\{
  \begin{array}{ll}
    1 & (i = n) \\
    0 & (i > n)
  \end{array}
\right..
\end{eqnarray*}
This implies that the canonical pairing $\rC(\Zp,\Qp) \times \Zp \dbrack{\Zp} \to \Qp$ sends the pair of $\binom{x}{n} \in \rC(\Zp,\Qp)$ and $([1]-[0])^i \in \Zp \dbrack{\Zp}$ to $1$ if $i = n$ and to $0$ if $i \neq n$. Since the algebra structure of the topological Hopf $\Zp$-algebra $\Zp \dbrack{\Zp}$ corresponds to the coalgebra structure of the Banach Hopf $\Qp$-algebra $\rC(\Zp,\Qp)$, the Mahler basis also possesses information of the comultiplication helpful in the study of Banach comodules over $\rC(\Zp,\Qp)$, which correspond to $p$-adic representations of $\Gm$. Further observations of an $\NBAb$-orthonormal Schauder $\Qp$-linear basis of the dual non-commutative Banach Hopf $\Qp$-algebra of the non-cocommutative topological Hopf $\Zp$-algebra associated to a non-commutative $p$-adic Lie group of higher dimension are given in \cite{Mih21} \S 4.1. Concerning this work, the interesting observation is that the non-Archimedean Banach rings $\Int(\Z_{\triv})$ and $\rC(\Zp,\Z_{\triv})$ are both $(\NBAb,\Z_{\triv})$-forms of $\rC(\Zp,\Q_p)$ but they are non-isomorphic as Banach rings, and the interplay between these two different presentations of $\rC(\Zp,\Q_p)$ is an important theme of its study. To appreciate the difference between the algebras $\Int(\Z_{\triv})$ and $\rC(\Zp,\Z_{\triv})$ one can notice that they behave very differently for the base change functors. Indeed, consider a prime $\ell \ne p$, then $\Int(\Z_{\triv}) \wh{\otimes}_{\Z_{\triv}} \Q_\ell \cong \rC(\Z_\ell,\Q_\ell)$ and $\rC(\Zp,\Z_{\triv}) \wh{\otimes}_{\Z_{\triv}} \Q_\ell \cong \rC(\Zp,\Q_\ell)$.
\end{exm}

\vspace{0.1in}
Similarly, it is reasonable to expect applications of the simpler $\Z$-form associated to a generalised van der Put basis. As the van der Put basis can be used to characterise Lipschitz functions (cf.\ \cite{Sch84} Theorem 63.2), $\rC^1$-functions with zero derivatives (cf.\ \cite{Sch84} Theorem 63.3), functions with higher differentiabilities (cf.\ \cite{Sch84} Exercise 63.D), and monotonous functions (cf.\ \cite{Sch84} Exercise 63.E), one can formally generalise those notions of smoothness to notions for a bounded continuous function $X \to \Qp$, which can be identified with a continuous function $\zeta(X) \to \Qp$ by the universality of Banaschewski compactification, with respect to a fixed generalised van der Put basis $S$ of $\rCfin(X,\Z_{\triv})$. Similarly, one can imitate the computation of the integration using the van der Put basis (cf.\ \cite{Sch84} Exercise 62.E (ii)) to formally define an integration of a bounded continuous function $X \to \Qp$ with respect to $S$ and a fixed assignment $S \to \Qp$ analogous to a $p$-adic measure.

\vspace{0.1in}
A subring $A \subset \rC(\Zp,\Qp)$ is said to {\it satisfy the $p$-adic Weierstrass approximation theorem} if $A \cap \Qp[T]$ is dense in $A$ with respect to the restriction of $\n{\cdot}_{\rC(\Zp,\Qp)}$. By the relations $\Int(\Z_{\triv}) \subset \Q[T]$ and $\rC(\Zp,\Z_{\triv}) \cap \Qp[T] = \Z$, it follows that $\Int(\Z_{\triv})$ satisfies the $p$-adic Weierstrass approximation theorem, but $\rC(\Zp,\Z_{\triv})$ does not satisfy the $p$-adic Weierstrass approximation theorem. Then, a natural question is: what approximation theorem should we expect for $\rC(\Zp,\Z_{\triv})$ (or more generally, for $\rCfin(X,\Z_{\triv})$) to hold? One natural candidate is the Stone--Weierstrass theorem.

\vspace{0.1in}
A subring $A \subset \rCbd(X,R)$ is said to {\it separate points of $X$} if for any $(x_0,x_1) \in X^2$ with $x_0 \neq x_1$, there exists an $f \in A$ such that $f(x_0) \neq f(x_1)$. An $R$-subalgebra $A \subset \rCbd(X,R)$ is said to {\it satisfy the Stone--Weierstrass theorem over $R$} if for any $R$-subalgebra $B \subset A$ separating points of $X$, $\set{f \in A}{R f \cap B \neq \ens{0}}$ is dense in $A$. We note that requiring $B$ itself to be dense in $A$ is not sufficient, as $\Z + p \rC(\Zp,\Z_{\triv})$ and $\sum_{n \in \N} \Z p^n 1_{\Zp,\Z_{\triv},n + p^n \Zp}$ are subrings of $\rC(\Zp,\Z_{\triv})$ separating points of $\Zp$ but are not dense in $\rC(\Zp,\Z_{\triv})$. We say that $X$ is zero dimensional if $\CO(X)$ generates the topology of $X$. We show an extension of \cite{Mih14} Lemma 4.17.

\begin{prp}
\label{non-Archimedean completely Hausdorff}
Suppose that $R \neq \ens{0}$, and $X$ is zero dimensional and $T_0$.
\begin{itemize}
\item[(i)] Every dense subring $A \subset \rCbd(X,R)$ separates points of $X$.
\item[(ii)] For any $\Z$-form $(A,\varphi)$ of $\rCbd(X,R)$ over $\Z$, the set $\set{\varphi(f \otimes 1)}{f \in A}$ is a subring of $\rCbd(X,R)$ separating points of $X$.
\end{itemize}
\end{prp}

\begin{proof}
We show the assertion (i). Let $(x_0,x_1) \in X^2$ with $x_0 \neq x_1$. Since $X$ is $T_0$, there exists an open subset $V \subset X$ such that $V \cap \ens{x_0,x_1}$ is a singleton. Since $X$ is zero-dimensional, there exists a $U \in \CO(X)$ such that $V \cap \ens{x_0,x_1} \subset U \subset V$. Then, $1_{X,R,U} \in \rCbd(X,R)$ separates $x_0$ and $x_1$ by $R \neq \ens{0}$. Since $A$ is dense in $\rCbd(X,R)$, there exists an $f \in A$ such that $\n{1_{X,R,U} - f}_{\rCbd(X,R)} < 2^{-1}$. We have $\n{f(x_0) - f(x_1)}_R > \n{1_{X,R,U}(x_0) - 1_{X,R,U}(x_1)} - 2^{-1} \times 2 = 0$ and hence $f(x_0) \neq f(x_1)$.

\vspace{0.1in}
We show the assertion (ii). Let $(x_0,x_1) \in X^2$ with $x_0 \neq x_1$. If $\varphi(f \otimes 1)(x_0) = \varphi(f \otimes 1)(x_1)$ for any $f \in A$, then $f(x_0) = f(x_1)$ for any $f \in A_R$, where $A_R \subset \rCbd(X,R)$ is the $R$-subalgebra generated by $\set{\varphi(f \otimes 1)}{f \in A}$. This contradicts the assertion (i), because $A_R$ is dense in $\rCbd(X,R)$.
\end{proof}

We note that every zero dimensional $T_0$ space is Hausdorff, as the proof of Proposition \ref{non-Archimedean completely Hausdorff} shows. We recall that \Gelfand's theory ensures that if $X$ is completely regular and Hausdorff, then closed $\R$-subalgebras of $\rCbd(X,\R)$ separating points of $X$ precisely correspond to equivalence classes of Hausdorff compactifications of $X$, and its non-Archimedean counterpart (cf.\ \cite{Mih14} Theorem 4.19) ensures that if $X$ is zero-dimensional and Hausdorff, then closed $k$-subalgebras of $\rCbd(X,k)$ separating points of $X$ precisely correspond to equivalence classes of totally disconnected Hausdorff compactifications of $X$ for any local field $k$. Therefore, when we expect the Stone--Weierstrass theorem to hold, we need to assume $X \in \CH$.

\begin{thm}
\label{non-Archimedean Stone--Weierstrass}
If $X \in \CH$, then $\rC(X,R)$ satisfies the Stone--Weierstrass theorem over $R$ in the following cases:
\begin{itemize}
\item[(i)] There exists an isomorphism $R \to (\R,\v{\cdot}_{\R}^{\epsilon})$ of monoid objects in $\BAb$ for some $\epsilon \in [0,\infty)$.
\item[(ii)] There exists an isomorphism $R \to k$ of monoid objects in $\BAb$ for some complete non-Archimedean valued field $k$.
\item[(iii)] There exists an isomorphism $R \to \Z_{\infty}$ or $R \to \Z_{\triv}$ of monoid objects in $\BAb$.
\item[(iv)] There exists a total order $\leq$ on the underlying ring of $R$ such that $(R,\leq)$ is an ordered ring, and $R$ is isolated at $0$.
\end{itemize}
\end{thm}

\begin{proof}
The claim for the condition (i) restricted to the case $\epsilon > 0$ is equivalent to the classical Stone--Weierstrass theorem, the claim for the condition (ii) is equivalent to non-Archimedean Stone--Weierstrass theorem (cf.\ \cite{Ber90} 9.2.5 Theorem), and the claims for the condition (i) restricted to the case $\epsilon = 0$ and the condition (iii) immediately follow from the claim for the condition (iv). We show the statement for the condition (iv). Let $B \subset \rC(X,R)$ be a subring separating points of $X$. It suffices to show that for any $U \in \CO(X)$, there exists an $a_U \in R \setminus \ens{0}$ such that $a_U 1_{X,R,U} \in B$. If $U = \emptyset$, then $a_U$ can be taken as $1$. Therefore, we may assume $U \neq \emptyset$.

\vspace{0.1in}
First, we show that for any $x \in X \setminus U$, there exists an $f_x \in B$ such that $x \in f_x^{-1}(\ens{0}) \subset X \setminus U$. Let $y \in U$. Since $B$ separates points of $X$, there exists an $f_{x,y} \in B$ such that $f_{x,y}(x) \neq f_{x,y}(y)$. Replacing $f_{x,y}$ by $f_{x,y} - f_{x,y}(x) \in B$, we may assume $f_{x,y}(x) = 0$. By $U \subset \bigcup_{y \in U} f_{x,y}^{-1}(R \setminus \ens{0})$, the compactness of $X$ ensures that there exists a finite subset $S_x \subset U$ such that $U \subset \bigcup_{y \in S_x} f_{x,y}^{-1}(R \setminus \ens{0})$. Put $f_x \coloneqq \sum_{y \in S_x} f_{x,y}^2 \in B$. We have $S_x \neq \emptyset$ by $U \neq \emptyset$. For any $z \in X$, we have $f_x(z) = \sum_{y \in S_x} f_{x,y}(z)^2 \geq 0$, and hence $f_x(z) = 0$ if and only if $f_{x,y}(z) = 0$ for any $y \in S_x$ because $(R,\leq)$ is an ordered ring. This implies $x \in f_x^{-1}(\ens{0}) = \bigcap_{y \in S_x} f_{x,y}^{-1}(\ens{0}) \subset X \setminus U$.

\vspace{0.1in}
Secondly, we show that for any $x \in X \setminus U$, there exists an $(a_x,e_x) \in R \times B$ such that $a_x > 0$, $x \in e_x^{-1}(\ens{0}) \subset X \setminus U$, and $e_x^2 = a_x e_x$. We have that $f_x(X)$ is a finite subset of $R$ by Proposition \ref{Cfin is closed} (ii), contains $0$ by $f_x(x) = 0$, and does not coincide with $\ens{0}$ by $U \neq \emptyset$. Put $a_x \coloneqq \prod_{a \in f_x(X) \setminus \ens{0}} a \in R$ and $e_x \coloneqq a_x - \prod_{a \in f_x(X) \setminus \ens{0}} (a - f_x) \in B$. Since every element of $f_x(X)$ is the sum of square elements, $f_x(X) \setminus \ens{0}$ consists of positive elements. This implies $a_x > 0$. We have $e_x(y) = 0$ for any $y \in f_x^{-1}(\ens{0})$ and $e_x(y) = a_x$ for any $y \in f_x^{-1}(f_x(X) \setminus \ens{0}) = X \setminus f_x^{-1}(\ens{0})$. This implies $x \in e_x^{-1}(\ens{0}) \subset X \setminus U$ and $e_x^2 = a_x e_x$.

\vspace{0.1in}
Finally, we show that there exists an $a_U \in R$ such that $a_U > 0$ and $a_U 1_{X,R,U} \in B$. By $X \setminus U = \bigcup_{x \in X \setminus U} e_x^{-1}(\ens{0})$, the compactness of $X$ ensures that there exists a finite subset $S \subset X \setminus U$ such that $X \setminus U = \bigcup_{x \in S} e_x^{-1}(\ens{0})$. Put $a_U \coloneqq \prod_{x \in S} a_x \in R$. Then, we have $a_U > 0$ and $a_U 1_{X,R,U} = \prod_{x \in S} e_x \in B$.
\end{proof}

We note that the total order $\leq$ in the condition (iv) in Theorem \ref{non-Archimedean Stone--Weierstrass} is used only to restrict the ring structure of $R$, and does not have to be compatible with $\n{\cdot}_R$ in any sense. For example, $\R_{\triv}$ and $\R_{\hyb}$ satisfy the condition (iv), although the canonical total order on $\R$ is irrelevant to $\n{\cdot}_{\triv}$ or $\n{\cdot}_{\hyb}$.

\vspace{0.3in}
\addcontentsline{toc}{section}{Acknowledgements}
\noindent {\Large \bf Acknowledgements}
\vspace{0.2in}

\noindent

\vspace{0.1in}
The first listed author thanks Kobi Kremnizer and Oren Ben-Bassat for important discussions concerning derived analytic geometry. During the preparation of this paper the first listed author has been supported by the DFG research grant BA 6560/1-1 ``\emph{Derived geometry and arithmetic}''.

\vspace{0.1in}
The second listed author thanks Fr\'ed\'eric Paugam for informing me of the interesting result in \cite{BK20} by the first listed author. We have started the joint work \cite{BM21} thanks to this information, and is continuing this joint work. The second listed author also thanks Takuma Hayashi for instructing him on elementary facts on derived categories when he was working on our preceding paper \cite{BM21}, and colleagues in universities for daily discussions. The second listed author is also thankful to his family for their deep affection. During the preparation of this paper, the second listed author has been supported by JSPS KAKENHI Grant Number 21K13763.

\end{document}